\documentclass[11pt, twoside, leqno]{amsart}  
\usepackage{lipsum}
\usepackage{amsfonts}
\usepackage{graphicx}
\usepackage{epstopdf}
\usepackage{algorithmic}
\usepackage{calligra}
\usepackage{amsfonts,amsmath,amsthm,amssymb,bm}
\usepackage{mathtools}
\usepackage{hyperref,cleveref}
\usepackage{autonum}
\usepackage{hhline}
\usepackage{array}
\usepackage{diagbox}
\usepackage{tcolorbox}
\usepackage{mdframed}
\usepackage{multicol}
\usepackage{graphicx}
\usepackage{subcaption}
\usepackage{moreverb}
\usepackage{bbm}
\usepackage[margin=1.38in]{geometry}
\usepackage{todonotes}
\usepackage{scalerel,amssymb}
\allowdisplaybreaks
\usepackage{mathrsfs}  
\usepackage{lineno}
\usepackage{todonotes}
\usepackage{tikz}
\usepackage{pgfplots}
\usepackage{siunitx}
\usepackage[numbers,sort&compress]{natbib}
\definecolor{mygreen}{HTML}{43a047}
\usepackage{subcaption}
\usepackage{doi}
\usepackage{appendix}
\usepackage{enumitem}

\definecolor{darkcyan}{rgb}{0.0, 0.5, 0.5}
\newcommand{\corrBK}[1]{{#1}}

\def\kersig{k_\sigma}
\def\kertrsig{k_{\tr\,\sigma}}
\def\kereps{k_\epsi}
\def\kertreps{k_{\treps}}
\def\KerTen{\mathbb{D}}

\def\bfy{\mathbf{y}}
\def\bfz{\mathbf{z}}
\def\bff{\mathbf{f}}
\def\bfg{\mathbf{g}}
\def\bfh{\mathbf{h}}
\def\bfp{\mathbf{p}}
\def\bfq{\mathbf{q}}
\def\bfu{\mathbf{u}}
\def\bfv{\mathbf{v}}
\def\bfw{\mathbf{w}}

\def\bfxi{\bm{\xi}}
\def\bfchi{\bm{\chi}}
\def\bfphi{\boldsymbol{\phi}}
\def\bfsigma{\boldsymbol{\sigma}}

\def\bfH{\boldsymbol{H}}
\def\bfU{\boldsymbol{U}}
\def\bfP{\boldsymbol{P}}

\def\bbA{\mathbb{A}} 

\def\bbC{\mathbb{C}}

\def\bbI{\mathbb{I}}
\def\R{\mathbb{R}}

\def\olbfp{\overline{\mathbf{p}}}

\def\GD{\Gamma_{\textup{D}}}
\def\GN{\Gamma_{\textup{N}}}

\def\div{\textup{div}}
\def\tr{\textup{tr}}
\def\epsi{\mathbf{\varepsilon}}
\def\treps{\textup{tr}\,\epsi}
\def\epsiu{\mathbf{\varepsilon(\mathbf{u})}}

\def\hat{\widehat}
\newcommand{\LT}[1]{\reallywidehat{#1}}
\newcommand{\ILT}[1]{\mathcal{L}^{-1}\left(#1\right)}
\newcommand{\lt}{\omega}

\newcommand{\ddt}{\frac{\textup{d}}{\textup{d}t}}

\newcommand{\dt}{\, \textup{d} t}
\newcommand{\ds}{\, \textup{d} s }
\newcommand{\dr}{\, \textup{d} r }
\newcommand{\dx}{\, \textup{d} x}
\newcommand{\dxs}{\, \textup{d}x\textup{d}s}
\newcommand{\dxt}{\, \textup{d}x\textup{d}t}
\newcommand{\dsr}{\, \textup{d}s\textup{d}r}

\newcommand{\intO}{\int_{\Omega}}

\newtheorem{lemma}{Lemma}
\newtheorem{proposition}{Proposition}
\newtheorem*{assumption*}{Assumptions}

\newtheorem{remark}{Remark}
\numberwithin{lemma}{section}
\numberwithin{proposition}{section}
\numberwithin{theorem}{section}
\numberwithin{equation}{section}
\makeatletter
\newcommand{\leqnomode}{\tagsleft@true}
\newcommand{\reqnomode}{\tagsleft@false}
\makeatother

\usepackage{scalerel,stackengine}
\stackMath
\newcommand\reallywidehat[1]{%
	\savestack{\tmpbox}{\stretchto{%
			\scaleto{%
				\scalerel*[\widthof{\ensuremath{#1}}]{\kern-.6pt\bigwedge\kern-.6pt}%
				{\rule[-\textheight/2]{1ex}{\textheight}}
			}{\textheight}%
		}{0.5ex}}%
	\stackon[1pt]{#1}{\tmpbox}%
}

\definecolor{grey}{rgb}{0.5,0.5,0.5}

\title[Determining kernels in linear viscoelasticity]{Determining kernels in linear viscoelasticity}


\author[B.\ Kaltenbacher, U.\ Khristenko, V.\ Nikoli\'c, M.\ L.\ Rajendran, and B.\ Wohlmuth]{\small Barbara Kaltenbacher, Ustim Khristenko, Vanja Nikoli\'c, Mabel Lizzy Rajendran, and Barbara Wohlmuth}

\address{  \small
	Department of Mathematics, 
	Alpen-Adria-Universit\"at Klagenfurt 
	\\ Universit\"atsstra\ss e 65--67, A-9020 Klagenfurt, Austria}
\email{barbara.kaltenbacher@aau.at}
\address{\small 
	Department of Mathematics, 
	Technical University of Munich   \\ 
	Boltzmannstra\ss e 3, 	
	85748 Garching, Germany}
\email{khristen@ma.tum.de} 
\address{\small Department of Mathematics, Radboud University   \\ 
	Heyendaalseweg 135,
	6525 AJ Nijmegen, The Netherlands}
\email{vanja.nikolic@ru.nl} 

\address{\small 
	Department of Mathematics, 
	Technical University of Munich   \\ 
	Boltzmannstra\ss e 3, 	
	85748 Garching, Germany}
\email{rajendrm@ma.tum.de} 
\address{\small 
	Department of Mathematics, 
	Technical University of Munich   \\ 
	Boltzmannstra\ss e 3, 	
	85748 Garching, Germany}
\email{wohlmuth@ma.tum.de} 

\begin{document}
\vspace*{8mm}
\begin{abstract}  
In  this  work,  we investigate  the  inverse  problem  of  determining  the  kernel  functions  that  best  describe  the mechanical  behavior of a complex medium modeled by a general nonlocal viscoelastic wave equation. To this end, we minimize a tracking-type data misfit function under this PDE constraint. We perform the well-posedness analysis of the state and adjoint problems and, using these results, rigorously derive the first-order sensitivities. Numerical experiments in a three-dimensional setting illustrate the method.   
\end{abstract}
\vspace*{-7mm}
\maketitle           
    \section{Introduction}
\par  When elastic waves propagate through complex media such as biological tissues, their amplitude is attenuated according to a frequency-dependent law. In recent years, it has become evident that these attenuation laws are more complicated than initially thought and that nonlocal wave equations are needed to model such behavior. Indeed, elastic wave equations with weakly singular kernels arise in various important medical and industrial applications of wave propagation in complex media. For example, frequency power laws with powers between zero and two are often encountered in shear-wave elastography; see~\cite{sinkus2007mr,holm2014comparison,klatt2007noninvasive}. \\ 
 \indent In  this  work,  we tackle  the  inverse  problem  of  determining  the  kernel  functions  that  best  describe  the mechanical  behavior of a complex medium from over-specified data. To this end, we first derive and analyze a general nonlocal elastic wave model containing three kernels, cf. \eqref{constitutive:3kernels} below.
By taking the Laplace transform, it can be seen that, in general, not all three kernels can be uniquely determined independently of each other. Therefore, we consider the inverse problem of determining two of these kernels (cf. \eqref{constitutive:3kernels} below) from over-specified data and study the PDE-constrained optimization problem resulting from the (optionally regularized) minimization of the data misfit. Using an adjoint-based calculation, we then rigorously derive the first-order sensitivities and develop a numerical kernel recovery method.\\
 \indent We note that some work in the direction of kernel identification can be found in the literature; in particular, we refer to~\cite{buterin2006inverse, durdiev2020problem,janno1997inverse,SlodickaSeliga2017, von2008specific} and the references contained therein. For results with Abel integral kernels, that is, models involving time-fractional derivatives, we refer to, e.g., \cite{HatanoNakagawaWangYamamoto:2013,JinKian:2021, LiYamamoto:2015,LiZhangJiaYamamoto:2013,RundellZhang:2017}.
 \par The remaining of our exposition is organized as follows. In Section~\ref{Sec:modelling}, we derive a general nonlocal elastodynamic wave equation and discuss the setup of the inverse problem. Section~\ref{Sec:forward_analysis} is concerned with the well-posedness analysis of the state problem. In Section~\ref{Sec:Adjoint}, we derive and analyze the corresponding adjoint problem. Section~\ref{Sec:Inverse} is focused on the estimation of the kernels from additional observations. Finally, in Section~\ref{Sec:numerics}, we present numerical experiments that illustrate our theoretical results.

\section{Modeling and problem setup}\label{Sec:modelling}
\par In this section, we provide the setup for the problem and derive the viscoelastic equation from the conservation law and the constitutive relation of stress and strain. \\
\indent Consider a bounded domain $\Omega\subset \R^d$, $d \in \{1, 2, 3\}$, with Lipschitz smooth boundary $\partial \Omega$ acted upon by force. Let  $\bfu=\bfu(x,t)$ be the displacement field and $\rho=\rho(x)$ the density. The linearized strain due to the deformation is defined by
$$\epsi(\bfu) = \frac12 \nabla \bfu + \nabla \bfu^T.$$
The balance of momentum in the body is given by 
\begin{equation}\label{conservation:momentum}
\rho\bfu_{tt} = \div\bfsigma+\bff,
\end{equation}
where $\bfsigma$ denotes the stress tensor.
\corrBK{
In this paper we consider the following three- and two-kernel constitutive relations involving hydrostatic and deviatoric parts of the strain $\tr\,\epsi(\bfu)= \sum_{i=1}^d\epsi(\bfu)_{ii}$, $\epsi_d(\bfu)=\epsi(\bfu)-\frac{1}{d}\mathbb{I}\tr\,\epsi(\bfu)$, separately 
\begin{equation}
\begin{aligned}
\bfsigma +(\kersig* \bfsigma)_t =\bbC \epsi(\bfu) + \kereps* \bbA\epsi(\bfu_t) + \kertreps * \bbI\tr\, \epsi(\bfu_t),\label{constitutive:3kernels}
\end{aligned}
\end{equation}
and 
\begin{equation}
\begin{aligned}
\bfsigma = \bbC\epsi(\bfu) + \kereps* \bbA\epsi(\bfu_t) + \kertreps * \bbI\tr\, \epsi(\bfu_t).
\label{constitutive:2kernels}
\end{aligned}
\end{equation}
We point to Remark~\ref{rem:models} for their relation to existing models in the literature.
The operator $*$ denotes the convolution on the positive half-line with respect to the time variable.
}

\subsection{The viscoelastic wave model} The model for the wave propagation in complex media considered in this work is derived by coupling the conservation law with more appropriate constitutive relation, as described above. The two hyperbolic models obtained by coupling \eqref{conservation:momentum} with two forms of constitutive relations, \eqref{constitutive:3kernels} and \eqref{constitutive:2kernels}, are as follows.
\begin{itemize}[leftmargin=0.5cm]
\item Model 1:
\begin{equation}\label{wave:3kernels}
\begin{aligned}
\rho\bfu_{tt}+\rho (\kersig*\bfu_{tt})_t-\div\left[\bbC\epsi(\bfu) + \kereps*\bbA\epsi(\bfu_t) + \kertreps*\bbI\tr\,\epsi(\bfu_t)\right]\\
= \bff + (\kersig* \bff)_t;
\end{aligned}
\end{equation}
\item Model 2:
\begin{equation}\label{wave:2kernels}
\begin{aligned}
\rho\bfu_{tt}-\div[\bbC\epsi(\bfu)+\kereps*\bbA\epsi(\bfu_t)
+\kertreps*\mathbb{I}\textup{tr}\,\epsi(\bfu_t)]=\bff.
\end{aligned} 
\end{equation}
\end{itemize}
We can rewrite both of these equations as
\begin{equation}
\begin{aligned}
\rho\bfz_{tt}-\div\left[\bbC \epsi(\bfz) +\KerTen*\epsi(\bfu_t) \right] =\bfg,
\end{aligned}
\end{equation}
by introducing the auxiliary variable $\bfz$, tensor $\KerTen$, and the forcing function $\bfg$: 
\begin{equation}\label{gf}
\begin{aligned}
\bfz=\bfu+\kersig*\bfu_t,\quad  
\KerTen=\bbA\kereps- \bbC\kersig{+ \bbI \kertreps \tr},\\
\bfg = \bff + (\kersig* \bff)_t
+\rho {\kersig}_t \cdot \bfu_t(0),
\end{aligned}
\end{equation}
where $\kersig=0$ for Model 2. We note that the scalar counterpart of this model has been considered in the literature with fractional kernels; see, for example,~\cite{kaltenbacher2021inverse}. We mention in passing that third-order in time models of viscoelasticity were introduced in~\cite{gorain2010stabilization} and have been recently a topic of extensive research, often referred to as the Moore--Gibson--Thompson viscoelasticity; see, for example,~\cite{pellicer2019optimal, conti2020analyticity} and the references therein.
\corrBK{
\subsection*{Assumptions on the medium parameters, kernels, and tensors} To allow for heterogeneous viscoelastic materials, we assume that $\rho \in L^\infty(\Omega)$ and that there exists $\underline{\rho}>0$, such that
\begin{equation}\label{ass:rho}
\rho(x) \geq \underline{\rho}>0 \qquad \text{a.e.\ in} \ \Omega.
\end{equation}
The fourth-order tensors $\bbC$ and $\bbA$ are assumed to be constant and symmetric,
\begin{equation}\label{ass:CA}
\bbC, \ \bbA \, \in \mbox{Sym}^4(\R^d),
\mbox{ and $\bbC $ is positive definite.} 
\end{equation}
Moreover,} 
we assume that the kernel $\kersig$ and tensor $\KerTen$ 
\corrBK{according to \eqref{gf}}
are such that the following non-negativity conditions hold:
\begin{align}
& {\int_0^t (\KerTen *\bfy)(s) : (\kersig*\bfy)_t(s) \ds \geq 0;} \label{AssumptionKernel1}\\
& \int_0^t (\KerTen *\bfy)(s) : \bfy(s) \ds\geq \underline{\gamma} \|\bfy\|^2_{H^{-\delta}(0,t)} \label{AssumptionKernel2}
\end{align}
for some $\underline{\gamma}>0$, $\delta >0$ 
\corrBK{ 
and all $\bfy\in W^{1,1}(0,T,\mbox{Sym}^4(\R^d))$. 
Here 
\begin{equation}\label{Sym}
\begin{aligned}
\mbox{Sym}^4(\R^d)&=\{\bbC\in \R^{d\times d\times d\times d}\, : 
\bbC_{ijk\ell}=\bbC_{ij\ell k}=\bbC_{jik\ell}=\bbC_{k\ell ij}\}\,,\\
\mbox{Sym}^2(\R^d)&=\{\bfy\in \R^{d\times d}\, : 
\bfy_{ij}=\bfy_{ji}\}\,,
\end{aligned}
\end{equation}
and $H^{-\delta}(0,t)$ is the dual of the Sobolev space $H_0^{\delta}(0,t)$.
}
In the scalar case \[\mathbb{D}(t)= \mathbb{M} m(t)\] with $\mathbb{M}\in \mbox{Sym}^4(\R^d)$ positive semidefinite and using the $\mathbb{M}$ bilinear form \[\langle \bfw,\bfv\rangle=\mathbb{M}\bfw:\bfv\] as well as seminorm, condition \eqref{AssumptionKernel1} follows from the integrated version of Lemma~\ref{lem:Alikhanov1}, provided $\kersig=k*m$ with $k\geq0$ and $k'\leq0$, using $\bfw=m*\bfy$, which satisfies {$\bfw(0)=0$ and assuming $\bfw \in H^{1}(0,T)$.}
In the same setting, assuming strict positivity of $\mathbb{M}$ and $\Re(\mathcal{F}m)(\omega)\geq\gamma(1+\omega^2)^{-\delta/2}$, $\omega\in\mathbb{R}$, 
\corrBK{(where $\Re$ denotes the real part and $\mathcal{F}$ the Fourier transform),} we can conclude 
\corrBK{from Lemma~\ref{lem:coercivityI}}
that
\eqref{AssumptionKernel2} holds.

\begin{remark}\label{rem:models}
\corrBK{
We comment on the relation of \eqref{constitutive:3kernels} and \eqref{constitutive:2kernels} to other models in the literature.
}
Assuming the material is homogeneous and isotropic, a very general relation between deviatoric and hydrostatic parts of stress ($\bfsigma_d,\tr\, \bfsigma$) and strain ($\epsi_d(\bfu), \tr\,\epsi(\bfu)$) is given by 
\begin{subequations}\label{constitutive:4kernels}
\begin{align}
\bfsigma_d + (\kersig * \bfsigma_d)_t &= 2 \mu \epsi_d(\bfu) + \kereps * \epsi_d(\bfu_t)\label{consitutive:4kernels_d}\\
\tr\,\bfsigma + (\kertrsig * \tr\,\bfsigma)_t &= (2\mu + d \lambda)\tr\, \epsi(\bfu)  +\kertreps * \tr\, \epsi(\bfu_t),\label{consitutive:4kernels_s}
\end{align}
\end{subequations}
where $\mu,\lambda$ are the Lam\'e constants and $\kersig,\kertrsig,\kereps,\kertreps$ are four scalar-valued kernels. 
\par Assuming $\epsi(\bfu)(0) = 0$ and using Laplace transform, 
\corrBK{
$$
\begin{aligned}
(1+\omega\widehat{\kersig})\widehat{\bfsigma}_d &= 
(2 \mu  + \omega\widehat{\kereps})\epsi_d(\widehat{\bfu})\\
(1+\omega\widehat{\kertrsig}) \tr\,\widehat{\bfsigma} &= 
(2\mu + d \lambda\,+\widehat{\kertreps}) \tr\, \epsi(\widehat{\bfu}),
\end{aligned}
$$
as well as $\bfsigma=\bfsigma_d+\frac{1}{d}  \bbI  \tr\, \bfsigma$,
}
it is not difficult to see that \eqref{consitutive:4kernels_d} and \eqref{consitutive:4kernels_s} can be reduced to a form with three kernels non-uniquely. One such reduction is with the kernels $(\kersig,\kereps,\kertreps^{(1)})$, where
$$\kertreps^{(1)} = \ILT{\frac{1}{d(1+\lt \LT{\kertrsig})} 
\left[(2\mu+d\lambda)(\LT{\kersig}-\LT{\kertrsig})+\widehat{\kertreps}(1+\lt \LT\kersig)\right]},$$ 
\corrBK{
where skipping the superscript ${}^{(1)}$ we obtain the form \eqref{constitutive:3kernels}
in case of the special choice
\begin{equation}\label{CArem}
\mathbb{C}_{ijk\ell}=\delta_{ik}\delta_{j\ell}(2\mu +\lambda\delta_{ij}\textup{tr}),\quad
\mathbb{A}_{ijk\ell}=\delta_{ik}\delta_{j\ell}.
\end{equation}
}
\\
\indent Using the same technique once again, the consitutive relation \eqref{constitutive:3kernels} can be reduced to a form with only two kernels ($\kereps^{(2)},\kertreps^{(2)}$), where 
\begin{equation}\label{kertil}
\kereps^{(2)}= \ILT{\frac{ \LT{\kereps}-2\mu\LT{\kersig}}{1 +\lt\LT{\kersig}}},\ 
\kertreps^{(2)}=\ILT{\frac{\LT{\kertreps^{(1)}}-\lambda\LT{\kersig}}{1 +\lt\LT{\kersig}}},
\end{equation} 
which we again denote by $\kereps, \kertreps$ and obtain the form 
\eqref{constitutive:2kernels}
Existing works consider one of the above 
\corrBK{three forms \eqref{constitutive:3kernels}, \eqref{constitutive:2kernels}, or \eqref{constitutive:4kernels},} as constituent relation between stress and strain in 3D with either fractional kernels \[g_{\alpha}(t)=\frac{t^{\alpha-1}}{\Gamma(\alpha)}\] or the finite sum of these fractional kernels. The consitutive relation of the form \eqref{constitutive:4kernels} with $\kersig=\kereps,\kertrsig=\kertreps$ is used in, e.g,~\cite{enelund1997time,saedpanah2014well}. The form \eqref{constitutive:3kernels} with $\kersig=\kereps=\kertreps$ is studied in \cite{oparnica2020well}. The two-kernel representation form \eqref{constitutive:2kernels} is employed in \cite{alotta2017behavior} with fractional kernels. The four-kernel form  \eqref{constitutive:4kernels} with the kernels being sum of fractional kernels is considered in  \cite{schmidt2002finite}. 
\end{remark}

\subsection{The inverse problem}
We next discuss the inverse problem of determining the kernels from the measurements of the displacement $\bfu$ on finite number of points on the surface of the boundary. We use Tikhonov regularization, (see, e.g.,  \cite{EKN89,SeidmanVogel89,TikhonovArsenin77}) for recovering the unknown kernels from over-specified data and end up with a PDE-constrained optimization problem:
\begin{equation}\label{minJ}
\min_{\vec{k},\bfu\in X\times \bfU} J(\vec{k},\bfu),
\end{equation}
where $\vec{k}=(\kersig,\kereps,\kertreps)$, such that
\begin{equation}\label{PDEconstr}
\left \{ \begin{aligned}
&\rho\bfz_{tt}-\div\left[\bbC \epsi(\bfz) +\KerTen*\epsi(\bfu_t)\right] =\bfg,\quad && \text{in }\ \Omega \times (0,T),\\
& \bfz=\bfu+\kersig*\bfu_t, \quad && \text{in }\ \Omega \times (0,T),\\
& \bfu =0 \qquad && \text{on }\ \GD \times (0,T), \\
& \bfsigma \cdot n= \bfh && \text{on }\ \GN \times (0,T), \\
& (\bfu, \bfu_t) \vert_{t=0} = (\bfu_0, \bfu_1) && \text{in }\ \Omega,
\end{aligned} \right.
\end{equation}
is satisfied in a weak sense, with $\KerTen$ and $\bfg$ defined in \eqref{gf}. \\
\indent There exists a large variety of optimization applications in the context of viscoelasticity. For this reason, we here consider a general cost function $J$. In Section \ref{Sec:Inverse} below, we will specify some particular choices that can be made in the context of parameter identification problems arising from the estimation of the kernels from additional measurements. 

\subsection{Notation and theoretical preliminaries}\label{SubSec:preliminaries}
Before proceeding with the analysis, we shortly introduce the function spaces and analytical techniques which will be used in the following sections. 

\par We recall that we have assumed $\Omega\subset \R^d, d\in \{1,2,3\}$ to be an open, connected, and bounded set with Lipschitz regular boundary $\partial \Omega$. Furthermore, $\partial \Omega$ is the disjoint union of $\GD$ and $\GN$ 
\corrBK{with $\GD$ nonempty}.\\
\indent  As usual, we equip the Sobolev and Lebesgue spaces $W^{k,p}(\Omega)$ (special case of $p=2$, $H^{k}(\Omega)$) and $L^p(\Omega)$ on $\Omega$ with the norms $\|\cdot\|_{W^{k,p}(\Omega)}$ and $\|\cdot\|_{L^p(\Omega)}$; their vector-valued variants are denoted by $W^{k,p}(\Omega)^d$ and $L^p(\Omega)^d$. \\
\indent For notational brevity, we introduce the Sobolev space that incorporates homogeneous Dirichlet boundary conditions:
\[
\bfH^1_{\textup{D}}(\Omega)=\{\bfw \in H^1(\Omega)^d: \ \bfw=0 \ \text{on} \ \GD\}.
\]
\indent We  often use $x \lesssim y$ to denote $x \leq Cy$, where $C>0$ is a generic constant that may depend on the final time. 
\corrBK{
Throughout the paper $\langle \cdot, \cdot \rangle$ denotes the duality pairing between $(\bfH^1_{\textup{D}}(\Omega))^*$ and $\bfH^1_{\textup{D}}(\Omega)$, and $(\cdot, \cdot)_\Omega$ is the $L^2$-product on $\Omega$ for scalars, vectors and tensors.}
\subsection*{Auxiliary results} We recall the useful Leibniz integral rule, integration by parts, and the transposition identity:
\begin{equation} \label{diffconvol}
\begin{aligned}
(k*\bfw)_t(t)=\ddt\int_0^t k(s)\bfw(t-s)\ds =(k*\bfw_t)(t)+k(t) \bfw(0),\\
\quad k\in L^1(0,T), \quad \bfw\in W^{1,1}(0,T;X)\subseteq C(0,T;X),
\end{aligned}
\end{equation}
\begin{equation} \label{integbyparts}
\begin{aligned}
\int_0^T \bfw_t(t) \bfq(T-t)\dt
=\int_0^T \bfw(t)\bfq_t(T-t)\dt + \bfw(T)\bfq(0) - \bfw(0)\bfq(T),\\
\quad \bfq, \ \bfw\in W^{1,1}(0,T;X),
\end{aligned}
\end{equation}
\begin{equation} \label{integbyparts1}
\begin{aligned}
\int_0^T (k*\bfw)(t)\bfq(T-t)\dt = \int_0^T \bfw(t) (k*\bfq)(T-t)\dt,\\
\quad k\in L^1(0,T), \quad \bfq, \ \bfw\in L^2(0,T;X)
\end{aligned}
\end{equation}
for some Banach space $X$. The last identity is obtained by changing the order of integration as follows:
\begin{equation}
\begin{aligned}
&\int_0^T (k*\bfw)(t)\bfq(T-t)\dt = \int_0^T \int_0^t k(t-s)\bfw(s)\ds\,\bfq(T-t)\dt \\
&=  \int_0^T \bfw(s) \int_s^T k(t-s) \bfq(T-t)\dt\ds = \int_0^T \bfw(s) \int_0^{T-s}k(T-s-r) \bfq(r)\, \textup{d}r\textup{d}s\\
&= \int_0^T \bfw(T-\ell) \int_0^{\ell}k(\ell-r) \bfq(r)\,\textup{d}r \textup{d}\ell
= \int_0^T \bfw(T-\ell) (k*\bfq)(\ell)\, \textup{d}\ell\\
&= \int_0^T \bfw(t) (k*\bfq)(T-t)\dt.
\end{aligned}
\end{equation}

\section{Analysis of a nonlocal viscoelastic equation} \label{Sec:forward_analysis}
In this section, we analyze the state problem \eqref{PDEconstr} associated with our inverse problem, which can be written in terms of $\bfu$ only as follows:
\begin{equation}\label{forward}
\left \{ \begin{aligned}
&\rho\bfu_{tt}+\rho(\kersig*\bfu_{tt})_t &&\\
&\hspace*{6mm}-\div\left[\bbC\epsi(\bfu)+\kereps*\bbA\epsi(\bfu_t) + \kertreps *\tr\,\epsi(\bfu_t)\bbI\right] =\bfg  &&\text{in }\ \Omega \times (0,T),\\
& \bfu =0 \quad && \text{on }\ \GD \times (0,T), \\
& \left[\bbC\epsi(\bfu)+\kereps*\bbA\epsi(\bfu_t) + \kertreps *\tr\,\epsi(\bfu_t)\bbI\right] \cdot n= \bfh +(\kersig*\bfh)_t && \text{on }\ \GN \times (0,T), \\
& (\bfu, \bfu_t) \vert_{t=0} = (\bfu_0, \bfu_1) && \text{in }\ \Omega,
\end{aligned} \right.
\end{equation}
with $\bfg$ defined in \eqref{gf}, where the Neumann boundary condition on $\GN$ results from the traction condition $\bfsigma\cdot n =\bfh$. \\
\indent The well-posedness analysis follows in spirit the arguments in~\cite{oparnica2020well, saedpanah2014well} with the main novelty arises from handling the $\kersig$ term, which is not present in these references. In particular, we need to distinguish two cases in our well-posedness analysis based on whether the kernel $\kersig$ is singular or not. This condition will influence the regularity of $\bfu$.
\begin{proposition}\label{Pro:state:exis_uni1} 
Let assumptions
\corrBK{\eqref{ass:rho} and \eqref{ass:CA} hold.} Given $T>0$, let
$\kersig \in W^{1,1}(0,T)$ with
\[\kersig(t)\geq \underline{k}>-1,\] $\kersig$ monotonically decreasing,
and $\kereps$, $\kertreps \in L^1(0,T)$. Let also 
$\KerTen\in W^{1,1}(0,T;\corrBK{\mbox{Sym}^4(\R^d)})$ and the coercivity assumptions \eqref{AssumptionKernel1}--\eqref{AssumptionKernel2} on the involved kernels hold for all $t \in [0,T]$. Consider problem \eqref{forward} with initial conditions
\[
 (\bfu_0, \bfu_1)\in H^1(\Omega)^d \times L^2(\Omega)^d.
\]
Further, assume that $\bfg \in L^1(0,T; L^2(\Omega)^d)$ and \[\bfh+(\kersig*\bfh)_t \in W^{1,1}(0,T; H^{-1/2}(\GN)^d).\] 
	Then there exists a unique 
	\begin{equation}
	\begin{aligned}
	\bfu \in \bfU =\left\{\bfu \in L^\infty(0,T; \bfH^1_{\textup{D}}(\Omega)): \, \bfu_t \in L^\infty(0,T; L^2(\Omega)^d), 
	\, \rho\bfu_{tt} \in L^1(0,T;(\bfH^1_{\textup{D}}(\Omega))^*
\right\},
	\end{aligned}
	\end{equation}
	such that
	\begin{equation} \label{weak_form}
	\begin{aligned}
\begin{multlined}[t]
\langle \rho \bfu_{tt}+\rho (\kersig*\bfu_{tt})_t, \bfv \rangle+(\bbC \epsiu, \epsi(\bfv))_{\Omega}\\+(\kereps*\bbA \epsi(\bfu_t), \epsi(\bfv))_{\Omega} +(\kertreps*\tr\,\epsi(\bfu_t)\bbI, \epsi(\bfv))_{\Omega} \\=(\bfg, \bfv)_{\Omega}+(\bfh+(\kersig*\bfh)_t, \bfv)_{\GN},
\end{multlined}
	\end{aligned}
	\end{equation}	
for all $\bfv \in \bfH^1_{\textup{D}}(\Omega)^d$, $\bfw \in L^2(\Omega)^d$ a.e.\ in time, with $(\bfu, \bfu_t) \vert_{t=0} = (\bfu_0, \bfu_1)$.	Furthermore, the solution satisfies the estimate
\begin{equation} \label{low_energy_est1}
\begin{aligned}
& \|\bfu_t\|_{L^
	\infty(0,T; L^2(\Omega)^d)}^2 +  \|\bfu\|_{L^\infty(0,T; H^1(\Omega)^d)}^2 
+\underline{\gamma} \|\epsi(\bfu_t)\|^2_{H^{-\delta}(0,t; L^2(\Omega)^{2d})}\\
\lesssim&\, \begin{multlined}[t] \|\bfu_0\|_{H^1(\Omega)^d}^2+\|\bfu_1\|_{L^2(\Omega)^d}^2  + \|\bfg\|_{L^1(0,T;L^2(\Omega)^d)}^2 \\
+\|\bfh+(\kersig*\bfh)_t\|^2_{W^{1,1}(0,t;H^{-1/2}(\GN)^d)}. \end{multlined}
\end{aligned}
\end{equation}
If additionally $\bfu_1=0$ and $\bfg \in L^\infty(0,T; L^2(\Omega)^d)$, then $ \rho\bfu_{tt} \in L^\infty(0,T;(\bfH^1_{\textup{D}}(\Omega))^*$.
\end{proposition}
\begin{proof}
	The proof follows by employing a Galerkin discretization in space based on the eigenvectors $\{\boldsymbol{\phi}_j\}_{j \geq 1}$, where
	\[
	a(\boldsymbol{\phi}_j, \bfv)= \tilde{\lambda}_j (\rho \boldsymbol{\phi}_j, \bfv)_{L^2}, \quad \forall \bfv \in  \bfH^1_{\textup{D}}(\Omega)
	\]
with the bilinear form
\[
a(\boldsymbol{\phi}_j, \bfv)=\intO \bbC \epsi(\boldsymbol{\phi}_j): \epsi(\bfv)\dx;
\]	
cf.~\cite{oparnica2020well}. The basis can be chosen as orthonormal in 
\corrBK{the weighted Lebesgue space} 
$L^2_{\rho}(\Omega)^d$, as proven in~\cite[Lemma 4.1]{oparnica2020well}. Denote $V_n= \textup{span}\{\boldsymbol{\phi}_1, \ldots, \boldsymbol{\phi}_n\}$. We seek an approximate solution in the form
\[
\bfu^n =\sum_{j=1}^n \xi_j(t)\bfphi_j(x)
\]	
with $\bfu^n(0)$,  $\bfu_t^n(0)$ 
chosen as $L^2_{\rho}$ projections of $\bfu_0$, $\bfu_1$
onto $V_n$. The semi-discrete problem can then be rewritten as the system
\begin{equation}
\begin{aligned}
&\begin{multlined}[t]
\xi_{i, tt}(t)+ (\kersig*\xi_{i, tt})_t(t)+\tilde{\lambda}_i \xi_i(t)\\+\sum_{j=1}^n (\kereps*\xi_{j, t})(t)(\bbA \epsi(\bfphi_j),\epsi(\bfphi_i)) +\sum_{j=1}^n (\kertreps*\xi_{j, t})(t) (\tr\,\epsi(\bfphi_j)\bbI, \tr\,\epsi(\bfphi_i)) \end{multlined}\\
=&\,(\bfg(t), \bfphi_i)_{\Omega}+((\bfh+(\kersig*\bfh)_t)(t), \bfphi_i)_{\GN}
\end{aligned}
\end{equation}
for $i=1, \ldots, n$ and $t \in [0,T]$. In vector notation $\bfxi=[\xi_1 \ \ldots  \ \xi_n]^T$, we have 
		\begin{equation} \label{semidiscr_system}
		\begin{aligned}
		\bfxi_{tt} +  (\kersig*\bfxi_{tt})_t + \tilde{\bm{\lambda}}\bfxi + \mathbf{A}\kereps*\bfxi_t + \bm{\mu} \kertreps*\bfxi_t= \bm{b},
		\end{aligned}
		\end{equation}
		with  the right-hand side vector given by
\begin{equation}
\begin{aligned}
\bm{b}(t) = (\bfg(t), \bfphi_i)_{\Omega}+((\bfh+(\kersig*\bfh)_t)(t), \bfphi_i)_{\GN}.
\end{aligned}
\end{equation}
Further, we have introduced $\tilde{\bm{\lambda}}=[\tilde{{\lambda}}_1 \ \ldots  \ \tilde{{\lambda}}_n]$, $\mathbf{A}=[\mathbf{A}_{ij}]_{n \times n}$, and $\bm{\mu}=[\bm{\mu}_{ij}]_{n \times n}$ with 
\begin{equation}
\begin{aligned}
\mathbf{A}_{ij}=(\bbA \epsi(\bfphi_j),\epsi(\bfphi_i)), \quad \bm{\mu}_{ij}=(\tr\,\epsi(\bfphi_j)\bbI, \tr\,\epsi(\bfphi_i)).
\end{aligned}
\end{equation}
Note that we have $\kersig \in W^{1,1}(0,T) \hookrightarrow C[0,T]$. Denote $\bfchi=\bfxi_{tt}$. Then
\begin{equation}
(\kersig *\bfxi_{tt})_t=k_{\sigma t}*\bfxi_{tt}+k_\sigma(0)\bfxi_{tt}
\end{equation}
with $k_{\sigma t} \in L^1(0,T)$. We have
\begin{equation}
\begin{aligned}
\bfchi_t=1*\bfchi+\bfxi_1, \quad \bfchi=1*1*\bfchi+t\bfxi_1+\bfxi_0.
\end{aligned}
\end{equation}
Then since we have assumed that $\kersig(0)>-1$, we can rewrite the equation as
\begin{equation}
\begin{aligned}
&(1+\kersig(0))\bfchi+k_{\sigma t}*\bfchi+\tilde{\bm{\lambda}}(1*1)*\bfchi+\mathbf{A}(\kereps*1)*\bfchi+ \bm{\mu}(\kertreps*1)*\bfchi\\
=&\,\bm{b} +\tilde{\bm{\lambda}}(t \bfxi_1+\bfxi_0)+\mathbf{A}\kereps(t)\bfxi_1+\bm{\mu}\kertreps(t)\bfxi_1.
\end{aligned}
\end{equation}
We can then see this problem as a system Volterra integral equations of the second kind
\begin{equation} \label{Volterra_SD}
\begin{aligned}
(1+\kersig(0))\bfchi+\bm{K}*\bfchi=\tilde{\bm{b}}(t)
\end{aligned}
\end{equation}
with the kernel
\begin{equation}
\bm{K}=k_{\sigma t}+\tilde{\bm{\lambda}}(1*1)+\mathbf{A}(\kereps*1)+ \bm{\mu}(\kertreps*1)
\end{equation}
and the right-hand side
\begin{equation}
\tilde{\bm{b}}(t)=\bm{b}(t) +\tilde{\bm{\lambda}}(t \bfxi_1+\bfxi_0)+\mathbf{A}\kereps(t)\bfxi_1+\bm{\mu}\kertreps(t)\bfxi_1.
\end{equation}
According to existence theory for the Volterra integral equations of the second kind (see, e.g.~\cite[Ch.\ 2, Theorem 3.5]{gripenberg1990volterra}), there exists a unique $\bfchi \in L^1(0,T)$, which solves \eqref{Volterra_SD}. Combined with the imposed initial conditions, we recover a unique $\bfxi \in W^{2,1}(0,T)$ and thus $\bfu \in W^{2,1}(0,T; \bfH^1_{\textup{D}}(\Omega))$. Note that then the combined function $\bfz=\kersig*\bfu_t+\bfu$ due to \eqref{diffconvol} and the identities
\begin{equation}\label{identityzt}
\bfz_t=k_{\sigma t}*\bfu_t+(1+\kersig(0))\bfu_t,
\end{equation}
\begin{equation}\label{identityztt}
\bfz_{tt}=k_{\sigma t}*\bfu_{tt}+k_{\sigma t}*\bfu_t(0)+(1+\kersig(0))\bfu_{tt}
\end{equation}
has the regularity
$\bfz\in W^{2,1}(0,T; \bfH^1_{\textup{D}}(\Omega))$ as well, which we will use in the energy analysis below.\\[2ex]
\noindent \emph{Energy analysis.} Testing the semi-discrete problem given in terms of $\bfz=\kersig*\bfu_t+\bfu$:
\begin{equation} \label{weak_form_z_all}
\begin{aligned} 
&\begin{multlined}
(\rho \bfz_{tt}, \bfv)_{\Omega}+(\bbC \epsi(\bfz), \epsi(\bfv))_{\Omega}+(\KerTen*\epsi(\bfu_t), \epsi(\bfv))_{\Omega}\\=(\bfg, \bfv)_{\Omega}+(\bfh+(\kersig*\bfh)_t, \bfv)_{\GN}, 
\end{multlined}
\end{aligned}
\end{equation}
with $\bfz_t(t)=(\kersig*\bfu_t+\bfu)_t(t) \in V_n$ and integrating in space and time leads at first to the following identity:
\begin{equation} \label{1st_identity}
\begin{aligned}
&\begin{multlined}[t] \frac12 \|\rho^{1/2}\bfz_t(t)\|_{L^2(\Omega)^d}^2 + \frac12 \|\bbC^{1/2}\epsi(\bfz(t))\|_{L^2(\Omega)^d}^2 
+\int_0^t (\KerTen*\epsi(\bfu_t), \epsi(\bfz_t))_{\Omega}\ds \end{multlined} \\
=&\, \begin{multlined}[t]\frac12 \|\rho^{1/2}\bfz_t(0)\|_{L^2(\Omega)^d}^2 + \frac12 \|\bbC^{1/2}\epsi(\bfz(0))\|_{L^2(\Omega)^d}^2 
+\int_0^t(\bfg, \bfz_t)_{\Omega}\ds\\+\int_0^t(\bfh+(\kersig*\bfh)_t, \bfz_t)_{\GN}\ds.\end{multlined}
\end{aligned}
\end{equation}	
By H\"older's and Young's inequalities, we have
\[
\int_0^t(\bfg, \bfz_t)_{\Omega}\ds \leq  \left(\int_0^t \|\rho^{-1/2}\bfg\|_{L^2(\Omega)^d}\ds\right)^2+ \frac14 \sup_{s \in (0,t)} \|\rho^{1/2}\bfz_t(s)\|^2_{L^2(\Omega)^d}.
\]
Employing integration by parts with respect to time yields
\begin{equation}
\begin{aligned}
&\int_0^t(\bfh+(\kersig*\bfh)_t, \bfz_t)_{\GN}\ds
=\, (\bfh+(\kersig*\bfh)_t, \bfz)_{\GN}\Big\vert_0^t
-\int_0^t(\bfh_t+(\kersig*\bfh)_{tt}, \bfz)_{\GN}\ds.
\end{aligned}
\end{equation}
Thus,
\begin{equation}
\begin{aligned}
&\int_0^t(\bfh+(\kersig*\bfh)_t, \bfz_t)_{\GN}\ds\\
\leq&\, \begin{multlined}[t]
\|(\bfh+(\kersig*\bfh)_t)(t)\|_{H^{-1/2}(\GN)^d}\|\bfz(t)\|_{H^{1/2}(\GN)^d} 
+\|\bfh(0)\|_{H^{-1/2}(\GN)^d}\|\bfz(0)\|_{H^{1/2}(\GN)^d}\\
-\|\bfh_t+(\kersig*\bfh)_{tt}\|_{L^1(0,t;H^{-1/2}(\GN)^d)} \|\bfz\|_{L^\infty(0,t; H^{1/2}(\GN)^d)}.
\end{multlined}
\end{aligned}
\end{equation}
By recalling the definition of $\bfz$, we further have
\[
\begin{aligned}
\int_0^t (\KerTen*\epsi(\bfu_t), \epsi(\bfz_t))_{\Omega}\ds
 =\, \int_0^t (\KerTen*\epsi(\bfu_t),  (\kersig*\epsi(\bfu_t))_t+\epsi(\bfu_t))_{\Omega}\ds.
\end{aligned}
\]
We can employ our assumption \eqref{AssumptionKernel2} on the kernels, which guarantees that
\[
\int_0^t (\KerTen*\epsi(\bfu_t), \epsi(\bfu_t))_{\Omega}\ds \geq \underline{\gamma} \|\epsi(\bfu_t)\|^2_{H^{-\delta}(0,t; L^2(\Omega)^d)} .
\]
We can further use assumption  \eqref{AssumptionKernel1} to conclude that 
\begin{equation}
\begin{aligned}
&\int_0^t (\KerTen*\epsi(\bfu_t),  (\kersig*\epsi(\bfu_t))_t)_{\Omega}\ds  \geq 0.
\end{aligned}
\end{equation}
Thus, from \eqref{1st_identity}, by relying on the assumptions on $\rho$ and $\bbC$, we have the estimate
\begin{equation}  \label{est:z}
\begin{aligned}
& \|\bfz_t(t)\|_{L^2(\Omega)^d}^2 +  \|\bfz(t)\|_{H^1(\Omega)^d}^2 
+\underline{\gamma} \|\epsi(\bfu_t)\|^2_{H^{-\delta}(0,t; L^2(\Omega)^{2d})}\\
\lesssim&\, \begin{multlined}[t] \|\bfz_t(0)\|_{L^2(\Omega)^d}^2 + \|\bfz(0)\|_{H^1(\Omega)^d}^2 + \|\bfg\|_{L^1(0,T;L^2(\Omega)^d)}^2
\\+\|\bfh+(\kersig*\bfh)_t\|^2_{W^{1,1}(0,t;H^{-1/2}(\GN)^d)} \end{multlined}
\end{aligned}
\end{equation}
at first in a discrete setting.\\
\indent Applying the $L^\infty(0,T;L^2(\Omega)^d)$ norm to \eqref{identityzt} with monotonically decreasing kernel $\kersig$ and using the estimate
\begin{equation}\label{normLinftyut}
\begin{aligned}
&\|k_{\sigma t}*\bfu_t\|_{L^\infty(0,T;L^2(\Omega)^d)}\\
\leq&\, \|k_{\sigma t}\|_{L^1(0,T)} \|\bfu_t\|_{L^\infty(0,T;L^2(\Omega)^d)}
= (k_\sigma(0)-k_\sigma(t)) \|\bfu_t\|_{L^\infty(0,T;L^2(\Omega)^d)}
\end{aligned}
\end{equation}
under the assumption $\kersig(t)\geq \underline{k}>-1$ we get
\[
(1+\underline{k}) \|\bfu_t\|_{L^\infty(0,T;L^2(\Omega))} \leq \|\bfz_t\|_{L^\infty(0,T;L^2(\Omega))}
\]
Likewise, taking the $L^\infty(0,T;\bfH^1_{\textup{D}}(\Omega)^d)$ norm of $\bfz=\kersig*\bfu_t+\bfu=k_{\sigma\,t}*\bfu+\kersig(0)\bfu-\kersig\bfu(0)$, we conclude 
\[(1+\underline{k}) \|\bfu\|_{L^\infty(0,T;\bfH^1_{\textup{D}}(\Omega)^d)} \leq \|\bfz_t\|_{L^\infty(0,T;\bfH^1_{\textup{D}}(\Omega)^d)}+\|\kersig\|_{L^\infty(0,T)}\|\bfu_0\|_{\bfH^1_{\textup{D}}(\Omega)^d)}.\]
\indent Since $\KerTen\in W^{1,1}(0,T;\corrBK{\mbox{Sym}^4(\R^d)})$, we can obtain an additional estimate on $\bfz_{tt}$ as follows. For $\bfv \in \bfH^1_{\textup{D}}(\Omega)$ with $\|\bfv\|_{\bfH^1_{\textup{D}}(\Omega)} \leq 1$, we decompose it into $\bfv=\bfv^1+\bfv^2$, where $\bfv^1 \in \textup{span}\{\bfphi_1, \ldots, \bfphi_n\}$ and $\bfv^2 \in \textup{span}\{\bfphi_1, \ldots, \bfphi_n\}^{\bot}$; see, e.g.,~\cite[Ch.\ 7.2]{evans2010partial} for similar arguments. Then $\|\bfv^1\|_{\bfH^1_{\textup{D}}(\Omega)} \leq 1$
and
\begin{equation}
\begin{aligned}
(\rho \bfz_{tt}, \bfv)_{\Omega}=&\, (\rho \bfz_{tt}, \bfv^1)_{\Omega} \\=&\, -(\bbC \epsi(\bfz), \epsi(\bfv^1))_{\Omega}-(\KerTen*\epsi(\bfu_t), \epsi(\bfv^1))_{\Omega}+(\bfg, \bfv^1)_{\Omega}+(\bfh+(\kersig*\bfh)_t, \bfv^1)_{\GN}.
\end{aligned}
\end{equation}
From here, using the identity
\[
(\KerTen*\epsi(\bfu_t))(t)=(\KerTen_t*\epsi(\bfu))(t)+\KerTen(0)\epsi(\bfu(t))-\KerTen(t)\epsi(\bfu(0))
\]
we have
\begin{equation}
\begin{aligned}
\|\rho \bfz_{tt}\|_{L^{{p}}(0,T; (\bfH^1_{\textup{D}}(\Omega))^*)} \lesssim&\, \begin{multlined}[t] \|\bbC \epsi(\bfz)\|_{L^{p}(0,T; L^2(\Omega)^d)}+\|\KerTen_t* \epsi(\bfu)+\KerTen(0)\epsi(\bfu)\|_{L^{p}(0,T; L^2(\Omega)^d)} \\+ \|\bfg\|_{L^{p}(0,T;\bfH^1_{\textup{D}}(\Omega))^*}+\|\bfh+(\kersig*\bfh)_t\|_{{L^{p}}(0,T;H^{-1/2}(\GN)^d)}\\+\|\KerTen\|_{L^{p}(0,T)}\, \|\epsi(\bfu(0))\|_{L^2(\Omega)^d}
.\end{multlined}
\end{aligned}
\end{equation}
{With $p=1$ we get $\rho \bfz_{tt} \in L^1(0,T; (\bfH^1_{\textup{D}}(\Omega))^*)$. If additionally $\bfg\in L^\infty(0,T;\bfH^1_{\textup{D}}(\Omega))^*$ and $\bfu_t(t=0)=0$, then we can also use $p=\infty$.}\\
\indent Similarly to \eqref{normLinftyut}, by taking the $L^{{p}}(0,T;\bfH^1_{\textup{D}}(\Omega)^*)$ norm of \eqref{identityztt}, we conclude the same regularity for $\rho\bfu_{tt}$ as for $\rho\bfz_{tt}$. \\ 
\indent Usual compactness arguments then lead to a solution of the state problem. Uniqueness follows by proving that the only solution of the homogeneous problem
\begin{equation} \label{weak_form_z_all_zero}
\begin{aligned} 
&\begin{multlined}
\langle \rho \bfz_{tt}, \bfv \rangle+(\bbC \epsi(\bfz), \epsi(\bfv))_{\Omega}+(\KerTen*\epsi(\bfu_t), \epsi(\bfv))_{\Omega}=0
\end{multlined}
\end{aligned}
\end{equation}
is $\bfu=0$. To prove uniqueness, we cannot test the above weak form with $\bfz_t$ as in the semi-discrete setting due to its insufficient spatial regularity. Instead, following, e.g.,~\cite[Ch.\ 7.2]{evans2010partial}, we take
\begin{equation}
\begin{aligned}
\bfv= \begin{cases}
\displaystyle \int_t^s \bfz(\tau)\, \textup{d}\tau \ &\text{if} \ 0 \leq t \leq s, \\
0 \ &\text{if} \ s \leq t \leq T,
\end{cases}
\end{aligned}
\end{equation}
which satisfies $\bfv(t) \in \bfH^1_{\textup{D}}(\Omega)$ for each $0 \leq t \leq T$. Furthermore,
\[
\bfv_t=-\bfz, \quad 0 \leq t \leq s. 
\]
Since $\bfv(s)=0$ and (with zero initial data) $\bfz_t(0)=0$, taking this test function in \eqref{weak_form_z_all_zero} yields
\begin{equation}
\begin{aligned}
\int_0^s \left(\langle \rho \bfz_t, \bfz \rangle -(\bbC \epsi(\bfv_t), \epsi(\bfv))_{\Omega}+((\KerTen*\epsi(\bfu))_t, \epsi(\bfv))_{\Omega}\right)\dt =0,
\end{aligned}
\end{equation}
where we have also used that $\KerTen*\epsi(\bfu_t)=(\KerTen*\epsi(\bfu))_t-\KerTen(t)\epsi(\bfu(0))=(\KerTen*\epsi(\bfu))_t$. Integration by parts with respect to time yields
\begin{equation} \label{uniqueness_eq}
\begin{aligned}
\frac12 \|\rho^{1/2}\bfz(t)\|_{L^2(\Omega)^d}^2 + \frac12 \|\bbC^{1/2}\epsi(\bfv(0))\|_{L^2(\Omega)^d}^2 +\int_0^s(\KerTen*\epsi(\bfu), \epsi(\bfz))_{\Omega}\dt =0.
\end{aligned}
\end{equation}
Due to our assumptions on $\KerTen$, we know that
\begin{equation}
\begin{aligned}
\int_0^s(\KerTen*\epsi(\bfu), \epsi(\bfz))_{\Omega}\dt =&\, \int_0^s(\KerTen*\epsi(\bfu), \kersig*\epsi(\bfu_t)+\bfu)_{\Omega}\dt \\
=&\, \int_0^s(\KerTen*\epsi(\bfu), (\kersig*\epsi(\bfu))_t+\bfu)_{\Omega}\dt \geq 0.
\end{aligned}
\end{equation}
Thus, from \eqref{uniqueness_eq}, we conclude that $\bfz=0$ and thus \[\kersig*\bfu_t+\bfu = k_{\sigma t}*\bfu+\kersig(0)\bfu+\bfu=0.\] Together with $\bfu\vert_{t=0}=0$, this implies that $\bfu=0$.
\end{proof}
\noindent We next consider the well-posedness of the state problem for a singular $\kersig$. Note that the analysis will lead to the same estimate \eqref{est:z} of the combined quantity $\bfz$, but the resulting estimate for $\bfu$ crucially depends on whether $\kersig$ is singular or not. 
We also point out that even for a singular kernel $\kersig(t) \sim t^{-\gamma}$ as $t \rightarrow 0$, where $(\kersig*\bfu_{tt})_t$ behaves similary to a Riemann--Liouville derivative of order 
$\gamma$ of $\bfu_{tt}$, no initial condition on $\bfu_{tt}$ is imposed, which is in line with the theory of Riemann--Liouville fractional ODEs; see, e.g., \cite{KilbasSrivastavaTrujillo:2006,Podlubny:1999}.
\begin{proposition}\label{Pro:state:exis_uni2} 
	Given $T>0$, let assumptions \corrBK{\eqref{ass:rho}, \eqref{ass:CA} and}
assumptions \eqref{AssumptionKernel1}--\eqref{AssumptionKernel2} on the involved kernels hold. Assume that $\kersig {\in L^1(0,T)}$ is singular; that is, there exists $\gamma \in (0,1]$, such that \[\kersig(t) \sim t^{-\gamma} \ \text{as} \ t \rightarrow 0.\]  Furthermore, assume that $\kereps$, $\kertreps \in L^1(0,T)$. Consider problem \eqref{forward} with the initial conditions
	\[
	(\bfu, \bfu_t) \vert_{t=0} = (\bfu_0, \bfu_1) \in H^1(\Omega)^d \times \{0\}. 
	\]
	Further, assume that the source term and boundary data satisfy the assumptions of Proposition~\ref{Pro:state:exis_uni1}. Then there exists a unique 
		\begin{equation}
	\begin{aligned}
	\bfu \in \bfU =\left\{\bfu \in L^\infty(0,T; \bfH^1_{\textup{D}}(\Omega)): \, \bfu_t \in H^{\gamma-1}(0,T; H^1(\Omega)^d),\, \right. 
	\left.  \bfu_{tt} \in H^{\gamma-1}(0,T; L^2(\Omega)^d)\right\},
	\end{aligned}
	\end{equation}
	 such that \eqref{weak_form} 
	 holds, with $(\bfu, \bfu_t) \vert_{t=0} = (\bfu_0, \bfu_1)$. Furthermore, the solution satisfies the following estimate:
	 \begin{equation} \label{low_energy_est}
	 \begin{aligned}
	 & \|\bfu_t\|^2_{H^{\gamma-1}(0,T;H^1(\Omega)^d)}+
	 \|\bfu_{tt}\|^2_{H^{\gamma-1}(0,T;L^2(\Omega)^d)}
	 +\underline{\gamma} \|\epsi(\bfu_t)\|^2_{H^{-\delta}(0,t; L^2(\Omega)^{2d})}\\
	 \lesssim&\, \begin{multlined}[t] \|\bfu_0\|_{H^1(\Omega)^d}^2+\|\bfu_1\|_{L^2(\Omega)^d}^2  + \|\bfg\|_{L^1(0,T;L^2(\Omega)^d)}^2 
	 +\|\bfh+(\kersig*\bfh)_t\|^2_{W^{1,1}(0,t;H^{-1/2}(\GN)^d)}. \end{multlined}
	 \end{aligned}
	 \end{equation}
\end{proposition}
\begin{proof}
The proof in the singular case follows via a Galerkin approximation similarly to before. However, to prove the existence of the semi-discrete problem, since $\kersig$ is singular, we can now introduce $\ell$, such that $\ell * \kersig=1$ with $\ell \in L^1(0, \infty)$, according to Tauberian Theorems \cite[Theorems 2, 3, Chapter XIII.5]{Feller:1971}
\[ \begin{aligned}
\kersig(t)\sim t^{-\gamma}\mbox{ as }t\to0\ \Leftrightarrow \
\LT{\kersig}(s)\sim s^{\gamma-1}\mbox{ as }s\to\infty\ \Leftrightarrow \\
\LT{\ell}(s)=\frac{\LT{1}}{\LT{\kersig}(s)}=\frac{1}{s\LT{\kersig}(s)}\sim s^{-\gamma}\mbox{ as }s\to\infty\ \Leftrightarrow \ 
\ell(s)\sim t^{\gamma-1}\mbox{ as }t\to0.
\end{aligned}\]
We then set
	\begin{equation} \label{def_chi}
	\begin{aligned}
	\hat{\bfchi}=(\kersig *\bfxi_{tt})_t. 
	\end{aligned}
	\end{equation} 
	as the new unknown. From \eqref{def_chi}, we have
	\begin{equation}
	\bfxi_{tt}(t) =(\ell*\hat{\bfchi})(t) 
	\end{equation} 
	Then
	\begin{equation}
	\begin{aligned}
	\bfxi_t=&\,1*\bfxi_{tt}+\bfxi_1=1*\ell*\hat{\bfchi}+\bfxi_1, \\
	\bfxi=&\,1*\bfxi_t+\bfxi_0=1*1*\ell*\hat{\bfchi}+\bfxi_1 t+\bfxi_0
	\end{aligned}
	\end{equation}
	with $\bm{\xi}_0$ and $\bm{\xi}_1$, 
	being the coordinates of $\bfu^n(0)$ and $\bfu_t^n(0)$ 
	in the basis, respectively. Equation \eqref{semidiscr_system} results in the following equation for $\hat{\bfchi}$:
	\begin{equation}
	\begin{aligned}
	&\ell*\hat{\bfchi}+\hat{\bfchi}+\tilde{\bm{\lambda}}(1*1*\ell)*\hat{\bfchi}+\mathbf{A}(1*\kereps*\ell)*\hat{\bfchi}+\bm{\mu}(1*\kertreps*\ell)*\hat{\bfchi} \\
	=&\, \bm{b}+\tilde{\bm{\lambda}}(\bfxi_1t+\bfxi_0)+\mathbf{A}\kereps(t)\bfxi_1+\bm{\mu}\kertreps(t)\bfxi_1.
	\end{aligned}
	\end{equation}
	We can see this problem as a system of Volterra integral equations of second kind
	\begin{equation}
	\begin{aligned}
	\hat{\bfchi}+\bm{K}*\hat{\bfchi}=\tilde{\bm{b}}
	\end{aligned}
	\end{equation}
	with the kernel
	\begin{equation}
	\bm{K}=\ell+\tilde{\bm{\lambda}}(1*1*\ell)+\mathbf{A}(1*\kereps*\ell)+\bm{\mu}(1*\kertreps*\ell)
	\end{equation}
	and the right-hand side
	\begin{equation}
	\tilde{\bm{b}}(t)=\bm{b}+\tilde{\bm{\lambda}}(\bfxi_1t+\bfxi_0)+\mathbf{A}\kereps(t)\bfxi_1+\bm{\mu}\kertreps(t)\bfxi_1.
	\end{equation}
Similarly to before, existence theory for the Volterra integral equations of the second kind yields a unique $\hat{\bfchi} \in L^1(0,T)$ and thus $\bfz$, $\bfu \in W^{2,1}(0,T; \bfH^1_{\textup{D}}(\Omega))$.
\\

\noindent \emph{Energy analysis.} We can derive the uniform estimate \eqref{est:z} for $\bfz$ as before. In order to conclude additional regularity of $\bfu$, we use the fact that the auxiliary function $\bfv=\kersig*\bfu_t$, which vanishes at $t=0$, satisfies, along with its derivatives, the Volterra integral equations
\[
\bfv +\ell*\bfv =\bfz-\bfu_0\,, \quad
\bfv_t +\ell*\bfv_t =\bfz_t\,, \quad
\bfv_{tt} +\ell*\bfv_{tt} =\bfz_{tt}+\ell(t)\bfv_t(0).
\]
We then take the $L^\infty(0,T;\bfH^1_{\textup{D}}(\Omega))$, $L^\infty(0,T;L^2(\Omega)^d)$, and $L^1(0,T;(\bfH^1_{\textup{D}}(\Omega))^*)$ norms, respectively and use the estimate
\[\|\ell*\bfw\|_{L^\infty(0,T;X)}\leq\|\ell\|_{L^1(0,T)}\|\bfw\|_{L^\infty(0,T;X)}.\]
Assuming $\|\ell\|_{L^1(0,T)}< 1$, we can thus estimate 
\[
\begin{aligned}
(1-\|\ell\|_{L^1(0,T)}) \|\bfv\|_{L^\infty(0,T;\bfH^1_{\textup{D}}(\Omega))}
&\leq \|\bfz\|_{L^\infty(0,T;\bfH^1_{\textup{D}}(\Omega))}+\|\bfu_0\|_{\bfH^1_{\textup{D}}(\Omega)^d)}\,, \\
(1-\|\ell\|_{L^1(0,T)}) \|\bfv_t\|_{L^\infty(0,T;L^2(\Omega)^d)}
&\leq \|\bfz_t\|_{L^\infty(0,T;L^2(\Omega)^d)}\,, \\
(1-\|\ell\|_{L^1(0,T)}) \|\bfv_{tt}\|_{L^1(0,T;L^2(\Omega)^d)}
&\leq \|\bfz_{tt}\|_{{L^1}(0,T;(\bfH^1_{\textup{D}}(\Omega))^*)}+\|\ell\|_{L^1(0,T)}{\|\bfv_t(0)\|_{(\bfH^1_{\textup{D}}(\Omega))^*}}\\
&\leq 
\|\bfz_{tt}\|_{{L^1}(0,T;(\bfH^1_{\textup{D}}(\Omega))^*)}+{\|\bfz_t(0)\|_{(\bfH^1_{\textup{D}}(\Omega))^*}}
\end{aligned}
\]
To extract additional temporal regularity of $\bfu$ from these bounds, we consider 
$\kersig*\bfu_t=\bfv$ as an Abel integral equation for $\bfu_t$  and specify the assumption $\kersig\sim t^{-\gamma}$ for $t \rightarrow 0$ to $\kersig^\gamma(t):=\kersig(t) t^{-\gamma}$ satisfying $\kersig^\gamma(0)\not=0$, $\kersig^\gamma\in C[0,T]\cap W^{1,\infty}(0,T)$.
Then from, e.g., \cite[Theorem 1]{GorenfloYamamoto99} we can conclude
\[
\begin{aligned}
&\|\bfu_t\|^2_{H^{\gamma-1}(0,T;\bfH^1_{\textup{D}}(\Omega))}+
\|\bfu_{tt}\|^2_{H^{\gamma-1}(0,T;L^2(\Omega)^d)}\\
\lesssim&\, \begin{multlined}[t] \|\bfz_t(0)\|_{L^2(\Omega)^d}^2 + \|\bfz(0)\|_{H^1(\Omega)^d}^2 + \|\bfg\|_{L^1(0,T;L^2(\Omega)^d)}^2
+\|\bfh+(\kersig*\bfh)_t\|^2_{W^{1,1}(0,t;H^{-1/2}(\GN)^d)}. \end{multlined}
\end{aligned}
\]  
The rest of the arguments follow as in the previous proposition.\\
\corrBK{
Note that the use of a Hilbert space argument like \cite[Theorem 1]{GorenfloYamamoto99} leads to a loss as compared to some possibly most sharp result, and also the assumption of differentiability of $\kersig^\gamma$ can probably be relaxed.
}
\end{proof}
We note that higher spatial regularity of the solution can be obtained in case $\Gamma_{\textup{N}}=\emptyset$; a discussion of this case in included in Appendix~\ref{Appendix1}.

\section{The adjoint problem} \label{Sec:Adjoint}
We next wish to derive the adjoint problem. Recall that we are considering the optimization problem
\begin{equation}\label{minJ_constr}
\min_{\vec{k},\bfu\in X\times \bfU} J(\vec{k},\bfu) \mbox{  such that $\bfu$ solves \eqref{forward}},
\end{equation}
where $\vec{k}=(\kersig,\kereps,\kertreps)$. Since
\begin{equation}
	\begin{aligned}
		\bfg = \bff + (\kersig* \bff)_t,
	\end{aligned}
\end{equation}
with $\bfu_t(t=0)=0$, the Lagrange function reads 
\begin{equation}\label{Lagrange}
\begin{aligned}
\begin{multlined}[t]\mathcal{L}(\vec{k},\bfu,\bfp)= J(\vec{k},\bfu)\\
+\int_0^T\int_\Omega\Bigl(\Bigl( \bigl(\rho\bfu_{tt}+\rho (\kersig*\bfu_{tt})_t -\bff - (\kersig*\bff)_t\bigr)\cdot\bfp+\bigl(\bbC\epsi(\bfu)+\kereps*\bbA\epsi(\bfu_t)\bigr):\epsi(\bfp)
\\
\hspace*{4cm}+ \kertreps*\textup{tr}\epsi(\bfu_t))\,\textup{tr}\epsi(\bfp))\Bigr)\dx-\int_{\GN}(\bfh+(\kersig*\bfh)_t)\cdot\bfp\, \textup{d}S\Bigr)\dt \end{multlined}
\end{aligned} 
\end{equation}
and we use a solution space that incorporates the initial data and homogeneous  
boundary conditions $\bfu\in \{\bfu_0+t\bfu_1\}+\tilde{\bfU}$
\begin{equation}
\begin{aligned}
\tilde{\bfU} \subseteq\{\bfu\in &\bfU:\
(\bfu, \bfu_t) \vert_{t=0} = (0,0)
\}.
\end{aligned}
\end{equation}
Formally writing the first order optimality conditions
\[
\frac{\partial\mathcal{L}}{\partial\kersig}=0\,, \quad
\frac{\partial\mathcal{L}}{\partial\kereps}=0\,, \quad
\frac{\partial\mathcal{L}}{\partial\kertreps}=0\,, \quad
\frac{\partial\mathcal{L}}{\partial\bfu}=0\,, \quad
\frac{\partial\mathcal{L}}{\partial\bfp}=0\,, \quad
\]
at $(\vec{k},\bfu,\bfp)$, the fourth and third yield the state and adjoint equation, respectively. We will now focus on the second to last condition:
\begin{equation}\label{dLdu0}
\begin{aligned}
0&=\frac{\partial\mathcal{L}}{\partial\bfu}(\vec{k},\bfu,\bfp)
\bfv\\
&=\frac{\partial\mathcal{J}}{\partial\bfu}(\vec{k},\bfu)\bfv 
+\int_0^T\int_\Omega\Bigl( \bigl(\rho\bfv_{tt}+\rho (\kersig*\bfv_{tt})_t\bigr)\cdot\bfp
+ \bigl(\bbC\epsi(\bfv)+\kereps*\bbA\epsi(\bfv_t)\bigr):\epsi(\bfp)\\
&\hspace*{5cm} 
+\kertreps*\textup{tr}\epsi(\bfu_t))\,\textup{tr}\epsi(\bfp))\Bigr)
\Bigr)\dxt
\end{aligned}
\end{equation}
for all $\bfv \in \tilde{\bfU}$ which is supposed to uniquely determine the adjoint state $\bfp\in \tilde{\bfP}$ with the test space $\tilde{\bfP}\subseteq L^2(0,T; L^2(\Omega)^d)$ yet to be determined.
Knowing that the adjoint state $\bfp$ will solve a backwards in time PDE with end conditions at $T$, we make a timeflip right away and write $\bfp(t)=\olbfp(T-t)$. This allows us to use the elementary integration by parts and transposition identities \eqref{integbyparts} and \eqref{integbyparts1}. This method yields
\[
\begin{aligned}
&\int_0^T\int_\Omega\rho\bfv_{tt}(t)\cdot\olbfp(T-t)\dxt 
=\int_0^T\int_\Omega\rho\bfv_t(t)\cdot\olbfp_t(T-t)\dxt 
+\int_\Omega \rho\bfv_t(T)\cdot\olbfp(0)\dx\\
&= \int_0^T\int_\Omega\rho\bfv(t)\cdot\olbfp_{tt}(T-t)\dxt 
+\int_\Omega \rho\Bigl(\bfv(T)\cdot\olbfp_t(0)+\bfv_t(T)\cdot\olbfp(0)\Bigr)\dx
\end{aligned}
\]
since $\bfv_t(0)=0$, $\bfv(0)=0$, 
\[
\begin{aligned}
&\int_0^T\int_\Omega \rho (\kersig*\bfv_{tt})_t(t)\cdot\olbfp(T-t)\dxt\\
=&\,\int_0^T\int_\Omega \rho (\kersig*\bfv_{tt})(t)\cdot\olbfp_t(T-t)\dxt
+\int_\Omega \rho (\kersig*\bfv_{tt})(T)\cdot\olbfp(0)\dx\\
=&\,\int_0^T\int_\Omega \rho \bfv_{tt}(t)\cdot(\kersig*\olbfp_t)(T-t)\dxt
+\int_\Omega \rho (\kersig*\bfv_{tt})(T)\cdot\olbfp(0)\dx\\
=&\,\int_0^T\int_\Omega \rho \bfv_t(t)\cdot(\kersig*\olbfp_t)_t(T-t)\dxt
+\int_\Omega \rho (\kersig*\bfv_{tt})(T)\cdot\olbfp(0)\dx\\
=&\,\begin{multlined}[t]\int_0^T\int_\Omega \rho \bfv(t)\cdot(\kersig*\olbfp_t)_{tt}(T-t)\dxt
\\+\int_\Omega \rho \Bigl(\bfv(T)\cdot(\kersig*\olbfp_t)_t(0)
+(\kersig*\bfv_{tt})(T)\cdot\olbfp(0)\Bigr)\dx \end{multlined}
\end{aligned}
\]
due to the fact that $(\kersig*\bfv_{tt})(0)=0$, $(\kersig*\olbfp_t)(0)=0$, $\bfv_t(0)=0$, and $\bfv(0)=0$, where \[(\kersig*\olbfp_t)_t(0)=\lim_{t\to0+}\kersig(t)\olbfp_t(0).\] Therefore, in case of a singularity in $\kersig$ at $t=0$, we need to impose $\olbfp_t(0)=0$ and
\[
\begin{aligned}
&\int_0^T\int_\Omega (\kereps*\bbA\epsi(\bfv_t))(t):\epsi(\olbfp(T-t))\dxt
=\int_0^T\int_\Omega \bbA\epsi(\bfv_t(t)):(\kereps*\epsi(\olbfp))(T-t)\dxt\\
&=\int_0^T\int_\Omega \epsi(\bfv(t)):(\bbA\kereps*\epsi(\olbfp_t))(T-t)\dxt
\end{aligned}
\]
where we have used $(\bbA\kereps*\epsi(\olbfp_t))(0)=0$ and $\epsi(\bfv(0))=0$.

Furthermore, we define $\nabla_u J(\vec{k},\bfu)\in \tilde{\bfU}^*$ by the variational equation
\[
\int_0^T\int_\Omega\nabla_uJ(\vec{k},\bfu)(x,t) \cdot\bfv(x,t)\dxt = \frac{\partial\mathcal{J}}{\partial\bfu}(\vec{k},\bfu)\bfv
\quad \mbox{ for all }\bfv\in\tilde{\bfU}\,.
\]

Altogether, \eqref{dLdu0} becomes
\[
\begin{aligned}
0&=\int_0^T\int_\Omega\Bigl(\bfv(t)\cdot\bigl(\nabla_uJ(\vec{k},\bfu)(t)
+\rho\olbfp_{tt}(T-t)+\rho (\kersig*\olbfp_t)_{tt}(T-t)\bigr)\\
&\hspace*{2cm}+\epsi(\bfv(t)):\bigl(\bbC\epsi(\olbfp(T-t))+(\bbA\kereps*\epsi(\olbfp_t))(T-t)\bigr) \\
&\hspace*{4cm}+\textup{tr}\epsi(\bfv(t))\,(\kertreps*\textup{tr}\epsi(\olbfp_t))(T-t)
\Bigr)\dxt\\
&\quad + \int_\Omega\rho\Bigl(\bfv(T)\cdot\bigl(1+\lim_{t\to0+}\kersig(t)\bigr)\olbfp_t(0)
+\bigl(\bfv_t(T)+(\kersig*\bfv_{tt})(T)\bigr)\cdot\olbfp(0)
\Bigr)\, dx
\end{aligned}
\]
Since $\bfv$ is arbitrary, this implies that $\olbfp$ solves (in weak sense) the following problem:
\begin{equation}
\left \{ \begin{aligned}
&\rho\olbfp_{tt}+\rho (\kersig*\olbfp_t)_{tt}-\div[\bbC\epsi(\olbfp)+\kereps*\bbA\epsi(\olbfp_t)
+\kertreps*\textup{tr}\epsi(\olbfp_t)\mathbb{I}]\\
&\hspace*{8cm}=-\nabla_u\overline{J}(\vec{k},\bfu), \quad && \text{in }\ \Omega \times (0,T),\\
& \olbfp =0 \qquad && \text{on }\ \GD \times (0,T), \\
& [\bbC\epsi(\olbfp)+\kereps*\bbA\epsi(\olbfp_t)
+\kertreps*\textup{tr}\epsi(\olbfp_t)\mathbb{I}]\cdot n =0 \qquad && \text{on }\ \GN \times (0,T), \\
& (\olbfp, \olbfp_t) \vert_{t=0} = (0,0) && \text{in }\ \Omega,
\end{aligned} \right.
\end{equation}
with $\overline{J}(\vec{k},\bfu)(t)=J(\vec{k},\bfu)(T-t)$.

\medskip 

This way, we can not only characterize a minimizer of \eqref{minJ}, \eqref{PDEconstr}, but also compute the gradient of the reduced cost function
\[
j(\vec{k}) = J(\vec{k},\bfu(\vec{k}))\,, \ \text{ where $\bfu(\vec{k})$ solves \eqref{PDEconstr}}
\]
by using the fact that $j(\vec{k}) = \mathcal{L}(\vec{k},\bfu(\vec{k}),\bfp(\vec{k}))$ which by means of the Chain Rule results in
\begin{equation}\label{redgrad}
\begin{aligned}
\frac{\partial j}{\partial\kersig}(\vec{k})h 
=&\frac{\partial\mathcal{L}}{\partial\kersig}(\vec{k},\bfu(\vec{k}),\bfp(\vec{k}))h\\
&+\frac{\partial\mathcal{L}}{\partial\bfu}(\vec{k},\bfu(\vec{k}),\bfp(\vec{k}))\frac{\partial \bfu}{\partial\kersig}(\vec{k})h
+\frac{\partial\mathcal{L}}{\partial\bfp}(\vec{k},\bfu(\vec{k}),\bfp(\vec{k}))\frac{\partial \bfp}{\partial\kersig}(\vec{k})h\\
=&\frac{\partial\mathcal{L}}{\partial\kersig}(\vec{k},\bfu(\vec{k}),\bfp(\vec{k}))h.
\end{aligned}
\end{equation}
Here we have used the fact that $\bfu(\vec{k})$ and $\bfp(\vec{k})$ solve the state and adjoint equations, respectively; and similarly for 
$\frac{\partial j}{\partial\kereps}$, $\frac{\partial j}{\partial\kertreps}$.

Well-posedness of the adjoint problem follows from Proposition ~\ref{Pro:state:exis_uni2} in case of a singular kernel $\kersig$, provided $\nabla_u J(\vec{k},\bfu)\in L^1(0,T; L^2(\Omega)^d)$.

\section{Estimation of the kernels from additional observations} \label{Sec:Inverse}
By taking Laplace transforms, we have shown in Remark~\ref{rem:models} that at least in the isotropic case \eqref{CArem} the two- and three kernel formulations are equivalent.
Therefore, without loss of generality, concerning parameter identification we will focus on the two-kernel model (formally setting $\kersig=0$):
\begin{equation}\label{PDEconstr_kersig_zero}
\left \{ \begin{aligned}
&\rho\bfu_{tt}-\div[\bbC\epsi(\bfu)+\kereps*\bbA\epsi(\bfu_t)+\kertreps*\textup{tr}\epsi(\bfu_t)\mathbb{I}]=\bff, \quad && \text{in }\ \Omega \times (0,T),\\
& \bfu =0 \qquad && \text{on }\ \GD \times (0,T), \\
& [\bbC\epsi(\bfu)+\kereps*\bbA\epsi(\bfu_t)
+\kertreps*\textup{tr}\epsi(\bfu_t)\mathbb{I}]\cdot n =\bfh \qquad && \text{on }\ \GN \times (0,T), \\
& (\bfu, \bfu_t) \vert_{t=0} = (0,0)
&& \text{in }\ \Omega.
\end{aligned} \right.
\end{equation}

\medskip

Applying Tikhonov regularization for recovering the unknown kernels from overspecified data, we end up with a PDE-constrained optimization problem of the form \eqref{minJ} under the constraint \eqref{PDEconstr_kersig_zero}.

The derivation from Section~\ref{Sec:Adjoint} easily extends to the case of multiple PDE constraints \eqref{PDEconstr_kersig_zero} as a consequence of considering states resulting from multiple different excitations $\bff^n$, $n=1,\ldots N$, for example $N=2$ with  a pulling and a shearing experiment in order to distinguish between $\kereps$ and $\kertreps$. 
A typical example of a cost function (abbreviating 
$\vec{k}=(
\kereps,\kertreps)$) is then 
\begin{equation}\label{Jdelta}
\begin{aligned}
&J(\vec{k},\bfu_1,\ldots,\bfu_N)  \\
=&\,\frac12 \sum_{n=1}^N \sum_{i=1}^I \int_0^T|\bfu^n(x_i^n,t)-\bfu_i^{n,\textup{meas}}(t)|^2\dt 
+ \gamma_\epsilon \mathcal{R}_\epsilon(\kereps) 
+ \gamma_{\treps} \mathcal{R}_{\treps}(\kertreps)
\end{aligned} 
\end{equation}
for measurements $\bfu_i^{n,\textup{meas}}$ of state $\bfu^n$ at $I$ given points $x_i^n\in\overline{\Omega}$ (typically at the boundary), $i\in\{1,\ldots,I\}$. 
\corrBK{In case of limited regularity of $\bfu^n$ -- note that according to Propositions~\ref{Pro:state:exis_uni1}, \ref{Pro:state:exis_uni2}, we have $\bfu^n\in L^\infty(0,T; \bfH^1_{\textup{D}}(\Omega))$ -- point evaluations in space cannot be justified. We therefore consider a variant that locally averages over a boundary patch $\Gamma_i^n\subseteq\partial\Omega$ concentrated at $x_i^n$ and possibly weighted with some $L^\infty(\Gamma_i^n)$ function $\eta_i^n$: 
\begin{equation}\label{Jeta}
\begin{aligned}
&J(\vec{k},\bfu_1,\ldots,\bfu_N)  \\
=&\,\begin{multlined}[t]\frac12 \sum_{n=1}^N \sum_{i=1}^I \int_0^T|\int_{\Gamma_i^n} \eta_i^n(x)\bfu^n(x,t)\,ds(x)-\bfu_i^{\textup{meas}}(t)|^2\dt 
+ \gamma_\epsilon \mathcal{R}_\epsilon(\kereps) \\
+ \gamma_{\treps} \mathcal{R}_{\treps}(\kertreps) 
\,. \end{multlined}
\end{aligned} 
\end{equation}
For $J$ as in \eqref{Jdelta}, we have  
\[\nabla_uJ(\vec{k},\bfu_1,\ldots,\bfu_N)(x,t)= \sum_{i=1}^I \sum_{n=1}^N (\bfu(x_i,t)-\bfu_i^{\textup{meas}}(t))\delta(x-x_i),\] and for $J$ as in \eqref{Jeta}, we obtain
\[\nabla_uJ(\vec{k},\bfu_1,\ldots,\bfu_N)(x,t)=\sum_{i=1}^I \sum_{n=1}^N (\int_{\Gamma_i^n} \eta_i^n(x)\bfu^n(x,t)\,ds(x)-\bfu_i^{\textup{meas}}(t))\eta_i^n(x)\delta_{\Gamma_i^n}(x),\]
where for any $f\in W^{1,1}(\Omega)$, $\int_\Omega\delta_{\Gamma_i^n}(x)f(x)\, dx =\int_{\Gamma_i^n}(x)f(x)\, ds(x)$; see~\cite[Chapter 15]{Leoni:2009} for well-definedness of the trace.
}

As space for the kernels a canonical choice is $L^p(0,T)$, with $p>1$ close to one, since this space is still reflexive and also applicable to singular kernels as in Proposition~\ref{Pro:state:exis_uni2}. This choice allows for forward solutions that are regular enough to admit locally averaging measurements according to \eqref{Jeta} (but not point measurements as in \eqref{Jdelta} -- for this we would need $\bfu\in L^2(0,T;H^s(\Omega))$ with $s$ large enough so that $H^s(\Omega)$ embeds continuously into $C(\Omega)$).
Thus, well-definedness of the forward operator \[F:L^1(0,T)^2\to L^2(0,T;\mathbb{R}^{NI}),\quad 
\vec{k}\mapsto(t\mapsto(\int_\Omega \eta_i^n(x)\bfu^n(x,t))_{i\in\{1,\ldots,I\},n\in\{1,\ldots,N\}}))\] then allows us to apply regularization theory in (reflexive) Banach spaces; see, e.g., \cite{BurgerOsher04,Scherzeretal2009,SchusterKaltenbacherHofmannKazimierski:2012} and the references therein.
\corrBK{
To keep notation simple, we will restrict ourselves to $N=1$ when deriving the gradient of the reducted cost  function below.
}

In order to define a gradient, we need a Hilbert space setting and thus consider 
a weighted $L^2$ space $X_\epsilon=L^2_{w_{\epsilon}}(0,T)$ with a weight function $w_\epsilon$ that vanishes at $t=0$ (at a certain rate as $t\to0$) if $\kereps$ is expected to have a singularity that is stronger than $t^{-1/2}$. 
The simplest example of a regularization term is then just 
\begin{equation}\label{Rnorm}
\mathcal{R}_\epsilon(\kereps)=\frac12\|\kereps\|_{X_\epsilon}^2
\end{equation} 
but one might also make more sophisticated choices, such as 
\begin{equation}\label{Rinf}
\mathcal{R}_\epsilon(\kereps)=\inf_{\alpha_m\in(0,1),a_m\in\mathbb{R}} \Bigl(\frac12\|\kereps-\sum_{m=1}^M a_m\frac{t^{-\alpha_m}}{\Gamma(1-\alpha_m)}\|_{X}^2
+\beta\sum_{m=1}^M|a_m|\Bigr)
\end{equation} 
to promote closeness to a multi-term fractional derivative kernel with as few components as possible.
As a matter of fact, in our numerical experiments we did not need any regularization term to be added as the problem appears to be only midly ill-posed and so the regularization induced by discretization of the kernel was sufficient to deal with even high noise levels in the data.

Starting from the expression \eqref{redgrad}, we can then define the gradient of the reducted cost function by using the inner product 
$\langle k_1,k_2\rangle_{X_\epsilon}=\int_0^T w_\epsilon(t)k_1(t)k_2(t)\dt$
and the variational equation
\[
\int_0^T w_\epsilon(t)\nabla_{\kereps} j(\vec{k})(t) h(t)\dt
= \frac{\partial j}{\partial\kereps}(\vec{k})h 
\mbox{ for all }h\in X\,.
\]
Due to the fact that
\[
\begin{aligned}
&\frac{\partial\mathcal{L}}{\partial\kereps}(\vec{k},\bfu(\vec{k}),\bfp(\vec{k}))h 
= \int_0^T\int_\Omega (\bbA\epsi(\bfu_t)*\epsi(\olbfp))(T-t)\dxt 
+\gamma_\epsilon \mathcal{R}_\epsilon'(\kereps) h\\
&= \int_0^T h(s)\Bigl(\int_\Omega \rho\tau (\bfu_{tt}*\olbfp_t)(T-s)\,dx +\gamma_\epsilon w_\epsilon(s)\kereps(s)\Bigr)\ds,
\end{aligned}
\]
we have 
\[
\nabla_{\kereps} j(\vec{k})(t) = \gamma_\epsilon\nabla\mathcal{R}_\epsilon(\kereps)(t) 
+\frac{1}{w_\epsilon(t)} \int_\Omega  (\bbA\epsi(\bfu_t)*\epsi(\olbfp))(T-t)\dx,
\]
where the function $\nabla\mathcal{R}_\epsilon(\kereps)(t)$ is determined by the variational equation
\[
\int_0^Tw_\epsilon(t) \nabla\mathcal{R}_\epsilon(\kereps)(t) h(t)\dt
=\mathcal{R}_\epsilon'(\kereps) h \quad \mbox{ for all } h\in L^2_{w_\epsilon}(0,T).
\]
For example, in case \eqref{Rnorm}, 
$\mathcal{R}_\epsilon'(\kereps) h = \langle \kereps,h\rangle_{X_\epsilon}$ and $\nabla\mathcal{R}_\epsilon(\kereps)(t)=\kereps(t)$.

Note that in case of a strongly singular kernel $\kereps$, also the gradient will contain a singularity via the $\dfrac{1}{w_\epsilon(t)}$ term.

Likewise, we obtain
\[
\nabla_{\kertreps} j(\vec{k})(t) = \gamma_{\treps}\nabla\mathcal{R}_{\treps}(\kertreps)(t) 
+\frac{1}{w_{\treps}(t)} \int_\Omega  (\textup{tr}\epsi(\bfu_t)*\textup{tr}\epsi(\olbfp))(T-t)\dx.
\]

\begin{remark}[On uniqueness]
\corrBK{
To discuss uniqueness of identification of two kernels from two additional observations, we   
consider the equivalent reformulation in terms of the deviatoric and hydrostatic parts of the strain 
$$
\rho\bfu_{tt}-\div[\bbC\epsi_d(\bfu)+ \bfy\textup{tr}\epsi(\bfu_t)
+\kereps*\bbA\epsi_d(\bfu_t)+\widetilde{\kertreps}*\textup{tr}\epsi(\bfu_t)]=\bff
$$
with $\bfy=\frac{1}{d}\bbC \mathbb{I}$, $\widetilde{\kertreps}(t)=\kertreps(t) \mathbb{I} + \kereps(t)\bbA \mathbb{I}$ $\in \mbox{Sym}^2(\R^d)$ 
of the PDE in \eqref{PDEconstr_kersig_zero} 
$$
\rho\bfu_{tt}-\div[\bbC\epsi(\bfu)+\kereps*\bbA\epsi(\bfu_t)+\kertreps*\textup{tr}\epsi(\bfu_t)\mathbb{I}]=\bff
$$
and focus on the case of $\bbA$ and $\widetilde{\kertreps}$ being scalar multiples of $\bbC$ and $\bfy$, respectively. 
An example for this is the isotropic case 
\begin{equation}\label{isotropic}
(\bbC)_{ijkl}=2\mu\delta_{ik}\delta_{j\ell}, \ 
\bfy_{ij}=(\lambda+\frac{2\mu}{d})\delta_{ij}, \
(\bbA)_{ijkl}=\delta_{ik}\delta_{j\ell}.
\end{equation}
Additionally we assume to have have a separable source term so that we arrive at the form  
}
\begin{equation}\label{PDEC1C2}
\left \{ \begin{aligned}
&\rho\bfu_{tt}-\div[\bbC_1\epsi(\bfu)+\bbC_2\epsi(\bfu)&&\\ 
&\hspace*{0.5cm}+k_1*\bbC_1\epsi(\bfu_t)+k_2*\bbC_2\epsi(\bfu_t)]
=\ell(t)\bff(x)  && \text{in }\ \Omega \times (0,T),\\
& \bfu =0  && \text{on }\ \GD \times (0,T), \\
& [\bbC_1\epsi(\bfu)+\bbC_2\epsi(\bfu)+k_1*\bbC_1\epsi(\bfu_t)+k_2*\bbC_2\epsi(\bfu_t)]\cdot n =0  && \text{on }\ \GN \times (0,T), \\
& (\bfu, \bfu_t) \vert_{t=0} = (0,0)
&& \text{in }\ \Omega.
\end{aligned} \right.
\end{equation}
Moreover, we consider an idealized setting in which we can excite the system via two different eigenpairs $(\lambda_1,\varphi_1)$, $(\lambda_2,\varphi_2)$ of the operators $-\frac{1}{\rho}\div[\bbC_1\epsi]$ and $-\frac{1}{\rho}\div[\bbC_2\epsi]$, (equipped with mixed Dirichlet-Neumann boundary conditions) that are selfadjoint and positive with respect to the weighted $L^2$ inner product $(\bfv_1,\bfv_2)=\int_\Omega \rho \bfv_1\cdot\bfv_2\, dx$.
We assume the eigenfunctions to be orthogonal to each other as well as normalized with respect to this inner product and take measurements of the displacements resulting from excitations $\bff_i=f_i\rho\varphi_i$, $i=\{1,2\}$  
\[
y^i(t)=B\bfu^{i}(t) \quad t\in(0,T)\, \quad i\in\{1,2\}
\]
(e.g., $B\bfv=\bfv(x_0)$ for some boundary point $x_0$ 
or $B\bfv=\frac{1}{|\Gamma_0|}\int_{\Gamma_0}\bfv(x)\dx$ for some boundary patch $\Gamma_0\subseteq\partial\Omega$). Then testing \eqref{PDEC1C2} with $\varphi_k$, after integration by parts with respect to space and using the eigenvalue equation, as well as orthogonality $(\varphi_1,\varphi_2)=0$ yields
\[
u^i_{tt}+\lambda_i u^i+\lambda_i k_i* u^i_t=\ell(t)f_i \quad t\in(0,T)\, \quad i\in\{1,2\}
\]
for $u^i(t)=(\bfu(t),\varphi_i)$, and due to linearity we have $\bfu^{(i)}(t,x)=u^i(t) \varphi_i(x)$.
Assuming that $T=\infty$, we can apply the Laplace transform to obtain
\[
\mathcal{L}y^i(s) = \frac{f_i \, B\varphi_i \, \mathcal{L} \ell(s) }{s^2+\lambda_i \, s \, \mathcal{L}k_i(s) + \lambda_i}  \quad s\in\mathbb{C} 
\]
from which we further have
\[
\mathcal{L}k_i(s)
 = \frac{1}{\lambda_i\, s}\Bigl(\frac{f_i \, B\varphi_i \, \mathcal{L} \ell(s) }{\mathcal{L}y^i(s)}-s^2- \lambda_i\Bigr) \quad s\in\mathbb{C} 
\]
By injectivity of the Laplace transform, this yields uniqueness of $k_i$, $i=\{1,2\}$.

Of course, this setting is heavily idealized, but can be expected to be achieved -- at least approximately -- in the isotropic case \eqref{isotropic} when excitation is done via longitudinal and shear forces, respectively. This is the setup we consider in our numerical experiments.
\end{remark}

\section{Numerical Simulations}\label{Sec:numerics}

Let us consider a 3-dimensional rectangular viscoelastic beam of the size $1 \times 0.1 \times 0.04$.
The beam is clamped on the left end and excited on the right during the time interval $[0,t_{load}]$ with $t_{load}=0.8$ and then released.
The final time is $T=4$.
In our experiments, we consider two types of load: bending in $(x_1,x_2)$ plane and uniaxial extension ($x_1$-direction), respectively,
\begin{align}
	\bm{B}(t) &= [0,\; B_0\,t/t_{load},\; 0 ], \label{eq:load:bending}
	\\
	\bm{T}(t) &= [T_0\,t/t_{load},\; 0,\; 0], \label{eq:load:extension}
\end{align}
where the magnitudes are $B_0=1$ and $T_0=100$.
The Young's modulus is $10^3$ and the Poisson ratio is $0.3$.
There is no external volumetric force.
The beam is at rest at the initial time.
Thus, the dynamical system is given by the balance equation~\eqref{conservation:momentum} and the constitutive law~\eqref{constitutive:2kernels}:
\begin{equation}
\left \{ \begin{aligned}
&\rho\bfu_{tt} = \div\bfsigma,\\ 
&\bfsigma = \bbC\epsi(\bfu) + \kereps* \epsi_d(\bfu_t) + \kertreps * \bbI\tr\, \epsi(\bfu_t),\\
& \bfu\bigr|_{x_1=0}=0,
\qquad \bfsigma\cdot\bm{n}\bigr|_{x_1=1} = \bm{B}\text{ or }\bm{T}, \\
& (\bfu, \bfu_t, \bfsigma(\bfu)) \vert_{t=0} = (0,0,0).
\end{aligned} \right.
\end{equation}
We employ the Newmark method ($\beta=0.25$, $\gamma=0.5$) for time-stepping in combination with a sum-of-exponentials kernel representation; see, e.g., \cite{lam2020exponential,thiel2021numerical}. More precisely, we are looking for the kernels $k_{\epsi}$, $k_{\treps}$ in the form
\begin{equation}\label{key}
\kereps(t)  = \sum_{k=1}^{m_{\epsi}}w_{\epsi,k} e^{-\lambda_{\epsi,k} t},
\qquad
\kertreps(t) = \sum_{k=1}^{m_{\treps}}w_{\treps,k} e^{-\lambda_{\treps,k} t},
\end{equation}
with unknown weights $w_{\epsi,k}$, $w_{\treps,k}$ and exponents $\lambda_{\epsi,k}$, $\lambda_{\treps,k}$.
For simplicity, we have fixed $m_{\epsi}=m_{\treps}$.
Let us denote by $\theta := \{w_{\epsi,k}, w_{\treps,k}, \lambda_{\epsi,k},\lambda_{\treps,k}\}$ the vector of all learnable parameters.
We assume that the measurements for the tip displacement (more precisely, the mean displacement over the right end face) are available at some discrete time points in the interval $[0,t_{meas}]$ with $t_{meas}=T/2$. We minimize the objective function
\begin{equation}\label{key}
J(\theta) = \frac{1}{2}\|u_{tip}(\theta) - u_{tip}^{meas}\|^2_{[0,t_{meas}]},
\end{equation}
with respect to the parameters $\theta$, using the $\ell_2$-norm on the discretized interval $[0,t_{meas}]$.

We use the FEniCS platform~\cite{FEniCS} for the finite elements implementation of the forward problem and the dolfin-adjoint package~\cite{Mitusch2019} for the adjoint problem in combination with PyTorch~\cite{pytorch} through the Torch-FEniCS interface~\cite{Torch-FEniCS}.
The minimization is performed using the LBFGS method~\cite{wright1999numerical} with the strong Wolfe line search.
The beam is discretized using $60\times10\times5$ linear tetrahedral elements (\Cref{fig:beam}).
The time interval is discretized uniformly with the step size $\Delta t = T/100$.

Here, we have been assuming that the measurement data are given for all discrete time points of the time discretization scheme in the interval $[0,t_{meas}]$.
Moreover, in our synthetic setting, we assume that for a given $\theta_*$, it holds that
\begin{equation}\label{key}
u_{tip}^{meas} = u_{tip}(\theta_*) + \eta,
\end{equation}
where $\eta$ is an additive white Gaussian noise.
We consider in our numerical test cases different noise levels ranging from $2\%$ to $8\%$.

\begin{figure}[!ht]
	\includegraphics[width=\textwidth]{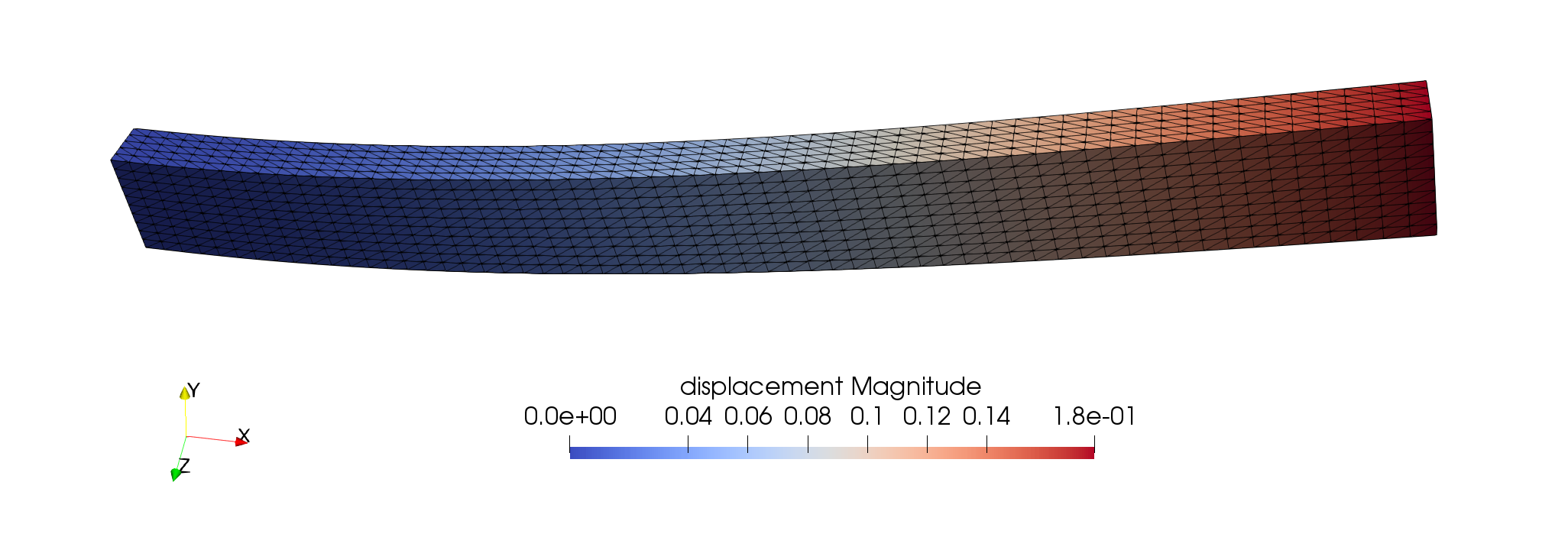}
	\caption{3D beam mesh. Color denotes the displacement magnitude at $t=t_{load}$ computed using the target kernel.}
	\label{fig:beam}
\end{figure}


First, we consider the case when $\kereps(t)=\kertreps(t)=k(t)$.
We define the target kernels $k(t)=k^{true}(t)$ as sum-of-exponentials being an accurate ($22$ modes)  approximation of the fractional kernel $t^{\alpha-1}/\Gamma(\alpha)$ with $\alpha=0.7$.
The expansion is obtained by a rational approximation of the Laplace spectrum using the AAA-algorithm~\cite{nakatsukasa2018aaa}. 
Using this kernel, we numerically generate synthetic measurements of the tip displacement on the interval $[0,t_{meas}]$, using only bending type load~\eqref{eq:load:bending}.
Then, we predict the kernel $k(t)=k^{pred}(t)$ only from these measurements.
To do this, we infer the parameters $\theta = \{w_k, \lambda_k\}$ minimizing $J(\theta)$.
We use a rational approximation (moderate accuracy with $8$ modes) of the fractional kernel $t^{\alpha_0-1}/\Gamma(\alpha_0)$ with $\alpha_0=0.5$ as initial guess.\\
\indent In~\Cref{fig:tip_displacement}, the resulting tip displacements are plotted on the complete interval $[0,T]$ for the models calibrated from noisy data with distinct noise levels.
We observe that fitting the noisy data on a shorter interval quite accurately results in a solution slightly deviating from the truth if considered on a larger time interval.
This can be explained by the fact that the measurement time interval $[0,t_{meas}]$ is not sufficiently large to fit accurately the tail of the kernel. Indeed, in~\Cref{fig:kernels}, where the resulting kernels are compared with the target, we can see that the prediction fits the target at mid-range but is less accurate in the tail.
Besides, we measured in $L^1([\Delta t,t_{meas}])$-norm the error between the target and the predicted kernels: $2\%$ noise - $0.032207$, $4\%$ noise - $0.085839$, $6\%$ noise - $0.144768$, $8\%$ noise - $0.159670$.\\
\indent The convergence of the loss function during the optimization process is depicted in~\Cref{fig:loss_convergence}.
There, we can see that for all the noise level cases, the calibration process converges after $8$ optimization steps.
Besides, the minimum of the loss function grows with the noise level.
The evolution of the energy in time is depicted in~\Cref{fig:energy} for the case of a $2\%$ noise level.
We observe that the energy of the calibrated model fits accurately the truth on the measurement interval $[0,t_{meas}]$.
Moreover, the calibrated kernel allows to predict the following evolution.

\begin{figure}[!ht]
	\begin{subfigure}{0.45\textwidth}
\begin{tikzpicture}

\begin{axis}[
axis background/.style={fill=white!93.3333333333333!black},
axis line style={white!73.7254901960784!black},
legend cell align={left},
legend style={
  fill opacity=0.8,
  draw opacity=1,
  text opacity=1,
  draw=white!80!black,
  fill=white!93.3333333333333!black
},
tick pos=left,
width=\textwidth,
x grid style={white!69.8039215686274!black},
xlabel={\(\displaystyle t\)},
xmajorgrids,
xmin=-0.158, xmax=4.198,
xtick style={color=black},
y grid style={white!69.8039215686274!black},
ylabel={Tip displacement},
ymajorgrids,
ymin=-0.102783378216812, ymax=0.20243957110851,
ytick style={color=black}
]
\addplot [thick, white!50.1960784313725!black]
table {%
0.04 0.000118245622784269
0.08 0.00062121067950784
0.12 0.00174373105008721
0.16 0.00366579387251918
0.2 0.00655908769417644
0.24 0.0105487347275054
0.28 0.0156920213840331
0.32 0.0219895634889271
0.36 0.0294014987415853
0.4 0.0378507035787044
0.44 0.047228726080069
0.48 0.0574053104663192
0.52 0.068238402874897
0.56 0.0795819015147969
0.6 0.09129265513684
0.64 0.103236642151981
0.68 0.115294053777973
0.72 0.127362727389473
0.76 0.139360342893173
0.8 0.151225376588065
0.84 0.160433796798694
0.88 0.163733178588812
0.92 0.161517476784078
0.96 0.154714721016651
1 0.143342868415664
1.04 0.128240714155524
1.08 0.11074571269518
1.12 0.0920327695704013
1.16 0.0729942824753509
1.2 0.0544926342540255
1.24 0.0372996008327564
1.28 0.022013847281947
1.32 0.00903946920135941
1.36 -0.001377168026714
1.4 -0.00913919528502781
1.44 -0.0142887695845025
1.48 -0.0169897318144086
1.52 -0.0174979021033599
1.56 -0.0161372629569025
1.6 -0.013274109190398
1.64 -0.0092952263765727
1.68 -0.004584411320193
1.72 0.000494517097781838
1.76 0.0056132914714162
1.8 0.0104892466516876
1.84 0.0148935271901225
1.88 0.0186536957463515
1.92 0.021654875871796
1.96 0.0238364824132723
2 0.0251879219153225
2.04 0.0257415217661726
2.08 0.0255652817145896
2.12 0.0247540218066475
2.16 0.0234212079779176
2.2 0.0216903764543967
2.24 0.019688012008555
2.28 0.0175365259941899
2.32 0.0153490836798642
2.36 0.0132250654137318
2.4 0.0112473351527346
2.44 0.00948033275987298
2.48 0.00796962967209247
2.52 0.00674217253984115
2.56 0.0058077427299394
2.6 0.0051607149881574
2.64 0.00478264404062891
2.68 0.00464481145333461
2.72 0.00471120411661611
2.76 0.00494114935555268
2.8 0.00529210261169478
2.84 0.00572187157402617
2.88 0.00619066105248366
2.92 0.00666258885662663
2.96 0.00710700953563177
3 0.00749915400923759
3.04 0.00782057021816032
3.08 0.00805908659334177
3.12 0.00820856887746814
3.16 0.0082682296006404
3.2 0.00824199048099535
3.24 0.00813764628892795
3.28 0.00796592824947189
3.32 0.00773941075596619
3.36 0.00747177752283447
3.4 0.00717695559251578
3.44 0.00686831397936787
3.48 0.00655816470722358
3.52 0.00625739477413983
3.56 0.00597491648730938
3.6 0.00571761888721662
3.64 0.00549032011954107
3.68 0.00529579931715787
3.72 0.00513490223048219
3.76 0.00500684264906361
3.8 0.00490941904037512
3.84 0.00483933103517193
3.88 0.00479248729765594
3.92 0.00476433706202038
3.96 0.00475013182360721
4 0.00474519323116127
};
\addlegendentry{initial}
\addplot [thick, red]
table {%
0.04 0.000141403212483305
0.08 0.000727139658382032
0.12 0.0019928765457489
0.16 0.00409305457398032
0.2 0.00720313478372924
0.24 0.011526405647341
0.28 0.017226035342861
0.32 0.0243462668879588
0.36 0.0328363380556222
0.4 0.0426152515055073
0.44 0.0535894573448172
0.48 0.0656303482561864
0.52 0.0785580874274052
0.56 0.0921545102766314
0.6 0.106195534914024
0.64 0.120469334733909
0.68 0.134774754578278
0.72 0.148918868075833
0.76 0.162725050090429
0.8 0.176046612294113
0.84 0.185805595647714
0.88 0.188397072851195
0.92 0.184892552655286
0.96 0.176756741893278
1 0.163421727502714
1.04 0.144186649060192
1.08 0.119774965229377
1.12 0.092624906928037
1.16 0.0648041851193486
1.2 0.0371220709912933
1.24 0.0100497706786711
1.28 -0.0154593412724
1.32 -0.0380860172017209
1.36 -0.0567448982481977
1.4 -0.0710101260412885
1.44 -0.0808495013476328
1.48 -0.0862303627190559
1.52 -0.0871011951960468
1.56 -0.0836093150128209
1.6 -0.0762244838434257
1.64 -0.0656363860710198
1.68 -0.0525651413729985
1.72 -0.0377018150640613
1.76 -0.021759709751901
1.8 -0.00550546573905605
1.84 0.0102974508979605
1.88 0.0249735589912001
1.92 0.0379764947668403
1.96 0.0488749512947745
2 0.0573378780152472
2.04 0.0631459821385958
2.08 0.0662145036531906
2.12 0.0665954934393178
2.16 0.0644575731248965
2.2 0.0600595141228581
2.24 0.0537339632550264
2.28 0.0458754490075709
2.32 0.0369234148547953
2.36 0.0273368747419296
2.4 0.0175692544388485
2.44 0.00804867523110216
2.48 -0.000835653248802475
2.52 -0.00874611208412305
2.56 -0.0154085819339795
2.6 -0.0206219046781395
2.64 -0.0242617088640482
2.68 -0.0262798440590519
2.72 -0.0267004232793765
2.76 -0.0256143766541088
2.8 -0.0231715929183793
2.84 -0.0195708280595693
2.88 -0.0150473431122194
2.92 -0.00986016348652699
2.96 -0.004279310072997
3 0.0014262672984476
3.04 0.00700017038167404
3.08 0.0122087112059183
3.12 0.0168497631512438
3.16 0.0207592803015302
3.2 0.0238158069637339
3.24 0.025942651165665
3.28 0.0271082929497722
3.32 0.0273247620074536
3.36 0.0266442797316815
3.4 0.0251542231344461
3.44 0.0229710068264361
3.48 0.0202330494778292
3.52 0.0170934295205748
3.56 0.0137122863461301
3.6 0.0102495288102531
3.64 0.0068578367112824
3.68 0.00367661942026999
3.72 0.000826865809486706
3.76 -0.00159279647083615
3.8 -0.00350886563335174
3.84 -0.00487442406372409
3.88 -0.00566943910291296
3.92 -0.00589977986175852
3.96 -0.00559527072699511
4 -0.0048067484040753
};
\addlegendentry{predict}
\addplot [thick, blue, dashed]
table {%
0.04 0.000140014946965442
0.08 0.000720146717751033
0.12 0.00197602267077028
0.16 0.00406630075681433
0.2 0.00717070410748714
0.24 0.0114893991065329
0.28 0.0171729726273926
0.32 0.024257281776499
0.36 0.0326987267658206
0.4 0.0424285762379663
0.44 0.0533561451606472
0.48 0.0653477803735181
0.52 0.0782216355847187
0.56 0.0917665500780245
0.6 0.105768777825507
0.64 0.120020810406467
0.68 0.13431974543689
0.72 0.148471557674555
0.76 0.162302327107032
0.8 0.175668855199718
0.84 0.185522230487431
0.88 0.18828550757443
0.92 0.184970468489367
0.96 0.176949786809227
1 0.163648549526998
1.04 0.144513464918043
1.08 0.120465058402529
1.12 0.0938621424309684
1.16 0.0664515627481474
1.2 0.0389259722493751
1.24 0.0119400428918729
1.28 -0.0133907463580785
1.32 -0.0358070142020983
1.36 -0.0544105602562812
1.4 -0.0688506245285249
1.44 -0.0789778438919373
1.48 -0.084635381533005
1.52 -0.085780796079188
1.56 -0.0826372626027194
1.6 -0.0756914754040303
1.64 -0.0655601174438489
1.68 -0.0528848704188393
1.72 -0.0383460996403683
1.76 -0.0226868952813502
1.8 -0.00668067250678404
1.84 0.00893899235555523
1.88 0.0235251969504032
1.92 0.0365269785765645
1.96 0.0474860865811452
2 0.0560516421229767
2.04 0.0619996785684976
2.08 0.0652405704349508
2.12 0.0658084515552972
2.16 0.063845892791152
2.2 0.0595906964641053
2.24 0.0533642771868648
2.28 0.0455538520527164
2.32 0.0365899952850025
2.36 0.0269230692316715
2.4 0.0170034009417749
2.44 0.00726331748549851
2.48 -0.00189933798223835
2.52 -0.0101368920634583
2.56 -0.0171639299778142
2.6 -0.0227654455213741
2.64 -0.0268005958154968
2.68 -0.0292039639538364
2.72 -0.0299836054833175
2.76 -0.029216308487116
2.8 -0.0270396011135873
2.84 -0.0236419008618425
2.88 -0.0192510248631362
2.92 -0.0141221757337677
2.96 -0.00852523851599646
3 -0.00273240247510174
3.04 0.00299372321993861
3.08 0.00841051249141467
3.12 0.0133044487767333
3.16 0.0174982532606038
3.2 0.0208561842723831
3.24 0.0232869690622002
3.28 0.0247448201422215
3.32 0.0252284369980071
3.36 0.024778225818181
3.4 0.0234717489664883
3.44 0.021418097623957
3.48 0.0187510972575327
3.52 0.01562205134893
3.56 0.0121921386114816
3.6 0.00862492875073749
3.64 0.00507921253767717
3.68 0.00170264485593567
3.72 -0.00137393298059126
3.76 -0.00404094126616981
3.8 -0.00621333745426983
3.84 -0.00783261341346062
3.88 -0.00886760117944158
3.92 -0.00931399360396492
3.96 -0.00919282589177759
4 -0.00854787571613503
};
\addlegendentry{truth}
\addplot [thick, black, dotted, mark=*, mark size=1, mark options={solid}]
table {%
0.04 0.00263455839575781
0.08 -0.002768834264949
0.12 0.00471665531148722
0.16 0.000819980442550852
0.2 0.00745865425091229
0.24 0.00541509551713187
0.28 0.0195955817378688
0.32 0.0221057928620117
0.36 0.0354855076615275
0.4 0.0407406194509883
0.44 0.054256087803833
0.48 0.0708944697672547
0.52 0.0779018478469223
0.56 0.0902436968352044
0.6 0.103912436507673
0.64 0.120331015659519
0.68 0.136297062207103
0.72 0.154469208660415
0.76 0.163846227630237
0.8 0.176918167176307
0.84 0.180600852687397
0.88 0.188565800684632
0.92 0.185405983077198
0.96 0.17696610601582
1 0.164644157589573
1.04 0.148431011685209
1.08 0.116062241869832
1.12 0.0898161018182711
1.16 0.0695625062745938
1.2 0.0374304940373879
1.24 0.00327812522441662
1.28 -0.0134804124349707
1.32 -0.0414345625533668
1.36 -0.0583063441177873
1.4 -0.0719437028006026
1.44 -0.0726716039230485
1.48 -0.0889096077929335
1.52 -0.084660514015973
1.56 -0.0793370163362308
1.6 -0.0791853933634988
1.64 -0.0639074255851893
1.68 -0.0580042547991485
1.72 -0.0342074043012132
1.76 -0.0246905562637477
1.8 -0.00617963819324681
1.84 0.00713240342334817
1.88 0.0285898504144725
1.92 0.0395600243399998
1.96 0.0472582098787436
2 0.0583131871358988
};
\addlegendentry{data}
\end{axis}

\end{tikzpicture}
		\caption{Noise level $2\%$}
	\end{subfigure}
	\hfill	
	\begin{subfigure}{0.45\textwidth}
\begin{tikzpicture}

\begin{axis}[
axis background/.style={fill=white!93.3333333333333!black},
axis line style={white!73.7254901960784!black},
legend cell align={left},
legend style={
  fill opacity=0.8,
  draw opacity=1,
  text opacity=1,
  draw=white!80!black,
  fill=white!93.3333333333333!black
},
tick pos=left,
width=\textwidth,
x grid style={white!69.8039215686274!black},
xlabel={\(\displaystyle t\)},
xmajorgrids,
xmin=-0.158, xmax=4.198,
xtick style={color=black},
y grid style={white!69.8039215686274!black},
ylabel={Tip displacement},
ymajorgrids,
ymin=-0.104290375265822, ymax=0.20407755483203,
ytick style={color=black}
]
\addplot [thick, white!50.1960784313725!black]
table {%
0.04 0.000118245622784269
0.08 0.00062121067950784
0.12 0.00174373105008721
0.16 0.00366579387251918
0.2 0.00655908769417644
0.24 0.0105487347275054
0.28 0.0156920213840331
0.32 0.0219895634889271
0.36 0.0294014987415853
0.4 0.0378507035787044
0.44 0.047228726080069
0.48 0.0574053104663192
0.52 0.068238402874897
0.56 0.0795819015147969
0.6 0.09129265513684
0.64 0.103236642151981
0.68 0.115294053777973
0.72 0.127362727389473
0.76 0.139360342893173
0.8 0.151225376588065
0.84 0.160433796798694
0.88 0.163733178588812
0.92 0.161517476784078
0.96 0.154714721016651
1 0.143342868415664
1.04 0.128240714155524
1.08 0.11074571269518
1.12 0.0920327695704013
1.16 0.0729942824753509
1.2 0.0544926342540255
1.24 0.0372996008327564
1.28 0.022013847281947
1.32 0.00903946920135941
1.36 -0.001377168026714
1.4 -0.00913919528502781
1.44 -0.0142887695845025
1.48 -0.0169897318144086
1.52 -0.0174979021033599
1.56 -0.0161372629569025
1.6 -0.013274109190398
1.64 -0.0092952263765727
1.68 -0.004584411320193
1.72 0.000494517097781838
1.76 0.0056132914714162
1.8 0.0104892466516876
1.84 0.0148935271901225
1.88 0.0186536957463515
1.92 0.021654875871796
1.96 0.0238364824132723
2 0.0251879219153225
2.04 0.0257415217661726
2.08 0.0255652817145896
2.12 0.0247540218066475
2.16 0.0234212079779176
2.2 0.0216903764543967
2.24 0.019688012008555
2.28 0.0175365259941899
2.32 0.0153490836798642
2.36 0.0132250654137318
2.4 0.0112473351527346
2.44 0.00948033275987298
2.48 0.00796962967209247
2.52 0.00674217253984115
2.56 0.0058077427299394
2.6 0.0051607149881574
2.64 0.00478264404062891
2.68 0.00464481145333461
2.72 0.00471120411661611
2.76 0.00494114935555268
2.8 0.00529210261169478
2.84 0.00572187157402617
2.88 0.00619066105248366
2.92 0.00666258885662663
2.96 0.00710700953563177
3 0.00749915400923759
3.04 0.00782057021816032
3.08 0.00805908659334177
3.12 0.00820856887746814
3.16 0.0082682296006404
3.2 0.00824199048099535
3.24 0.00813764628892795
3.28 0.00796592824947189
3.32 0.00773941075596619
3.36 0.00747177752283447
3.4 0.00717695559251578
3.44 0.00686831397936787
3.48 0.00655816470722358
3.52 0.00625739477413983
3.56 0.00597491648730938
3.6 0.00571761888721662
3.64 0.00549032011954107
3.68 0.00529579931715787
3.72 0.00513490223048219
3.76 0.00500684264906361
3.8 0.00490941904037512
3.84 0.00483933103517193
3.88 0.00479248729765594
3.92 0.00476433706202038
3.96 0.00475013182360721
4 0.00474519323116127
};
\addlegendentry{initial}
\addplot [thick, red]
table {%
0.04 0.000143549636164334
0.08 0.000736364621017585
0.12 0.00201352165316616
0.16 0.00412600129892021
0.2 0.00724589903296464
0.24 0.011581374475406
0.28 0.0173092898162955
0.32 0.0244812438492765
0.36 0.0330406185529733
0.4 0.0428990262637787
0.44 0.0539638022068975
0.48 0.0661127066176821
0.52 0.079168594325316
0.56 0.0929067496893584
0.6 0.107093345848913
0.64 0.1215122425935
0.68 0.135962843336045
0.72 0.150251688984689
0.76 0.164197380310327
0.8 0.177646983846285
0.84 0.187474584752092
0.88 0.190060830736673
0.92 0.186547117907172
0.96 0.178458962844433
1 0.165229231843129
1.04 0.145988429931061
1.08 0.121273071339917
1.12 0.0936438321555338
1.16 0.0654401738131413
1.2 0.0375099116278114
1.24 0.0101584101423984
1.28 -0.0157759467951273
1.32 -0.0389037047503858
1.36 -0.0579502218585684
1.4 -0.0724189491638711
1.44 -0.0823853200921954
1.48 -0.0879188042584641
1.52 -0.088950666748365
1.56 -0.085543771790919
1.6 -0.0781377069948307
1.64 -0.0674921223679884
1.68 -0.0544129410453594
1.72 -0.0396102252561249
1.76 -0.0237670304308738
1.8 -0.00763829417843325
1.84 0.00797864499886284
1.88 0.0223650185943867
1.92 0.0349672007355288
1.96 0.0453813757360393
2 0.0533036463350761
2.04 0.0585225380062685
2.08 0.0609557069080356
2.12 0.0606741162717506
2.16 0.0578830868726836
2.2 0.0528785173165565
2.24 0.0460179367572134
2.28 0.0377122694509837
2.32 0.0284188377873272
2.36 0.0186176538338877
2.4 0.00877862561702467
2.44 -0.000665004698136279
2.48 -0.00932920148408594
2.52 -0.0168863734407116
2.56 -0.0230747989705117
2.6 -0.0277102183048893
2.64 -0.0306918975867496
2.68 -0.0320007534280247
2.72 -0.0316916595125951
2.76 -0.0298852866573318
2.8 -0.0267598483879854
2.84 -0.0225411877010454
2.88 -0.0174893728825707
2.92 -0.0118838410968814
2.96 -0.00600896769911867
3 -0.000141879994012869
3.04 0.0054585809598718
3.08 0.0105619379774132
3.12 0.0149750253398926
3.16 0.0185483234549525
3.2 0.0211795792354285
3.24 0.0228146952748893
3.28 0.0234469535981619
3.32 0.0231143474180112
3.36 0.0218952020293435
3.4 0.0199019309860128
3.44 0.017273661483339
3.48 0.0141682127569404
3.52 0.0107540643334667
3.56 0.00720224703858392
3.6 0.00367867431515661
3.64 0.000336978381917746
3.68 -0.00268743110786547
3.72 -0.00528215499072595
3.76 -0.00736117346931019
3.8 -0.00886690316139874
3.84 -0.00977086536421898
3.88 -0.0100731034032091
3.92 -0.00980024211928659
3.96 -0.00900257518975153
4 -0.00775021497388213
};
\addlegendentry{predict}
\addplot [thick, blue, dashed]
table {%
0.04 0.000140014946965442
0.08 0.000720146717751033
0.12 0.00197602267077028
0.16 0.00406630075681433
0.2 0.00717070410748714
0.24 0.0114893991065329
0.28 0.0171729726273926
0.32 0.024257281776499
0.36 0.0326987267658206
0.4 0.0424285762379663
0.44 0.0533561451606472
0.48 0.0653477803735181
0.52 0.0782216355847187
0.56 0.0917665500780245
0.6 0.105768777825507
0.64 0.120020810406467
0.68 0.13431974543689
0.72 0.148471557674555
0.76 0.162302327107032
0.8 0.175668855199718
0.84 0.185522230487431
0.88 0.18828550757443
0.92 0.184970468489367
0.96 0.176949786809227
1 0.163648549526998
1.04 0.144513464918043
1.08 0.120465058402529
1.12 0.0938621424309684
1.16 0.0664515627481474
1.2 0.0389259722493751
1.24 0.0119400428918729
1.28 -0.0133907463580785
1.32 -0.0358070142020983
1.36 -0.0544105602562812
1.4 -0.0688506245285249
1.44 -0.0789778438919373
1.48 -0.084635381533005
1.52 -0.085780796079188
1.56 -0.0826372626027194
1.6 -0.0756914754040303
1.64 -0.0655601174438489
1.68 -0.0528848704188393
1.72 -0.0383460996403683
1.76 -0.0226868952813502
1.8 -0.00668067250678404
1.84 0.00893899235555523
1.88 0.0235251969504032
1.92 0.0365269785765645
1.96 0.0474860865811452
2 0.0560516421229767
2.04 0.0619996785684976
2.08 0.0652405704349508
2.12 0.0658084515552972
2.16 0.063845892791152
2.2 0.0595906964641053
2.24 0.0533642771868648
2.28 0.0455538520527164
2.32 0.0365899952850025
2.36 0.0269230692316715
2.4 0.0170034009417749
2.44 0.00726331748549851
2.48 -0.00189933798223835
2.52 -0.0101368920634583
2.56 -0.0171639299778142
2.6 -0.0227654455213741
2.64 -0.0268005958154968
2.68 -0.0292039639538364
2.72 -0.0299836054833175
2.76 -0.029216308487116
2.8 -0.0270396011135873
2.84 -0.0236419008618425
2.88 -0.0192510248631362
2.92 -0.0141221757337677
2.96 -0.00852523851599646
3 -0.00273240247510174
3.04 0.00299372321993861
3.08 0.00841051249141467
3.12 0.0133044487767333
3.16 0.0174982532606038
3.2 0.0208561842723831
3.24 0.0232869690622002
3.28 0.0247448201422215
3.32 0.0252284369980071
3.36 0.024778225818181
3.4 0.0234717489664883
3.44 0.021418097623957
3.48 0.0187510972575327
3.52 0.01562205134893
3.56 0.0121921386114816
3.6 0.00862492875073749
3.64 0.00507921253767717
3.68 0.00170264485593567
3.72 -0.00137393298059126
3.76 -0.00404094126616981
3.8 -0.00621333745426983
3.84 -0.00783261341346062
3.88 -0.00886760117944158
3.92 -0.00931399360396492
3.96 -0.00919282589177759
4 -0.00854787571613503
};
\addlegendentry{truth}
\addplot [thick, black, dotted, mark=*, mark size=1, mark options={solid}]
table {%
0.04 0.0166137947147885
0.08 -0.00162425060283778
0.12 -0.0040255236280863
0.16 0.00481024322941967
0.2 0.00673834974091081
0.24 0.0057428517145312
0.28 0.0168522929675824
0.32 0.0254690678594992
0.36 0.0296902696599192
0.4 0.0544936713446367
0.44 0.0614284258588205
0.48 0.0723830481423409
0.52 0.0835422061208028
0.56 0.0865390578861302
0.6 0.0842668657692423
0.64 0.11946044507611
0.68 0.135658995779456
0.72 0.145474197971888
0.76 0.172511756063818
0.8 0.185903588616642
0.84 0.188800981236581
0.88 0.180977811961572
0.92 0.181324674238366
0.96 0.178026151981149
1 0.170058442524249
1.04 0.137391213867391
1.08 0.139394293340992
1.12 0.0987708818637127
1.16 0.0484863693290145
1.2 0.0372231293858098
1.24 0.0185613481216107
1.28 -0.0111871704420976
1.32 -0.042664872747447
1.36 -0.067041571834822
1.4 -0.0669125837529158
1.44 -0.0902736511704655
1.48 -0.0781234368501319
1.52 -0.0880618548646067
1.56 -0.0853337996076662
1.6 -0.0881056724819281
1.64 -0.0669944000766336
1.68 -0.0490322679772017
1.72 -0.0323493746450829
1.76 -0.0327407597753244
1.8 0.000371025798785967
1.84 0.00138215111697797
1.88 0.0267238666104581
1.92 0.024910929001025
1.96 0.050870266217967
2 0.054361324356334
};
\addlegendentry{data}
\end{axis}

\end{tikzpicture}
		\caption{Noise level $4\%$}
	\end{subfigure}	

	\vspace{4ex}
	
	\begin{subfigure}{0.45\textwidth}
\begin{tikzpicture}

\begin{axis}[
axis background/.style={fill=white!93.3333333333333!black},
axis line style={white!73.7254901960784!black},
legend cell align={left},
legend style={
  fill opacity=0.8,
  draw opacity=1,
  text opacity=1,
  draw=white!80!black,
  fill=white!93.3333333333333!black
},
tick pos=left,
width=\textwidth,
x grid style={white!69.8039215686274!black},
xlabel={\(\displaystyle t\)},
xmajorgrids,
xmin=-0.158, xmax=4.198,
xtick style={color=black},
y grid style={white!69.8039215686274!black},
ylabel={Tip displacement},
ymajorgrids,
ymin=-0.107298984180571, ymax=0.213198865819995,
ytick style={color=black}
]
\addplot [thick, white!50.1960784313725!black]
table {%
0.04 0.000118245622784269
0.08 0.00062121067950784
0.12 0.00174373105008721
0.16 0.00366579387251918
0.2 0.00655908769417644
0.24 0.0105487347275054
0.28 0.0156920213840331
0.32 0.0219895634889271
0.36 0.0294014987415853
0.4 0.0378507035787044
0.44 0.047228726080069
0.48 0.0574053104663192
0.52 0.068238402874897
0.56 0.0795819015147969
0.6 0.09129265513684
0.64 0.103236642151981
0.68 0.115294053777973
0.72 0.127362727389473
0.76 0.139360342893173
0.8 0.151225376588065
0.84 0.160433796798694
0.88 0.163733178588812
0.92 0.161517476784078
0.96 0.154714721016651
1 0.143342868415664
1.04 0.128240714155524
1.08 0.11074571269518
1.12 0.0920327695704013
1.16 0.0729942824753509
1.2 0.0544926342540255
1.24 0.0372996008327564
1.28 0.022013847281947
1.32 0.00903946920135941
1.36 -0.001377168026714
1.4 -0.00913919528502781
1.44 -0.0142887695845025
1.48 -0.0169897318144086
1.52 -0.0174979021033599
1.56 -0.0161372629569025
1.6 -0.013274109190398
1.64 -0.0092952263765727
1.68 -0.004584411320193
1.72 0.000494517097781838
1.76 0.0056132914714162
1.8 0.0104892466516876
1.84 0.0148935271901225
1.88 0.0186536957463515
1.92 0.021654875871796
1.96 0.0238364824132723
2 0.0251879219153225
2.04 0.0257415217661726
2.08 0.0255652817145896
2.12 0.0247540218066475
2.16 0.0234212079779176
2.2 0.0216903764543967
2.24 0.019688012008555
2.28 0.0175365259941899
2.32 0.0153490836798642
2.36 0.0132250654137318
2.4 0.0112473351527346
2.44 0.00948033275987298
2.48 0.00796962967209247
2.52 0.00674217253984115
2.56 0.0058077427299394
2.6 0.0051607149881574
2.64 0.00478264404062891
2.68 0.00464481145333461
2.72 0.00471120411661611
2.76 0.00494114935555268
2.8 0.00529210261169478
2.84 0.00572187157402617
2.88 0.00619066105248366
2.92 0.00666258885662663
2.96 0.00710700953563177
3 0.00749915400923759
3.04 0.00782057021816032
3.08 0.00805908659334177
3.12 0.00820856887746814
3.16 0.0082682296006404
3.2 0.00824199048099535
3.24 0.00813764628892795
3.28 0.00796592824947189
3.32 0.00773941075596619
3.36 0.00747177752283447
3.4 0.00717695559251578
3.44 0.00686831397936787
3.48 0.00655816470722358
3.52 0.00625739477413983
3.56 0.00597491648730938
3.6 0.00571761888721662
3.64 0.00549032011954107
3.68 0.00529579931715787
3.72 0.00513490223048219
3.76 0.00500684264906361
3.8 0.00490941904037512
3.84 0.00483933103517193
3.88 0.00479248729765594
3.92 0.00476433706202038
3.96 0.00475013182360721
4 0.00474519323116127
};
\addlegendentry{initial}
\addplot [thick, red]
table {%
0.04 0.000141144856410368
0.08 0.000725672174315022
0.12 0.00198858537681626
0.16 0.0040842105907112
0.2 0.00718838990203087
0.24 0.0115041628817007
0.28 0.0171925087434446
0.32 0.0242943203300172
0.36 0.0327563934126299
0.4 0.0424965623515762
0.44 0.0534200187503029
0.48 0.0653959315074339
0.52 0.078241909534846
0.56 0.0917382571784026
0.6 0.105660824058266
0.64 0.11979835385553
0.68 0.133950194719219
0.72 0.147924041281309
0.76 0.161544591469391
0.8 0.17466735396613
0.84 0.184222349571502
0.88 0.186616347027425
0.92 0.18292290177674
0.96 0.174597635967269
1 0.161064017621246
1.04 0.141635233951994
1.08 0.117076146647403
1.12 0.0898512244147006
1.16 0.0620114513389819
1.2 0.0343401865590437
1.24 0.00730941510915596
1.28 -0.0181088612569975
1.32 -0.0405894552323116
1.36 -0.0590694251807083
1.4 -0.0731558321041581
1.44 -0.0828334557648619
1.48 -0.0880716985190713
1.52 -0.0888251235103745
1.56 -0.0852596114588255
1.6 -0.0778662542990371
1.64 -0.06734480992418
1.68 -0.0544129077119336
1.72 -0.0397562634302068
1.76 -0.02408581414567
1.8 -0.00816487689900359
1.84 0.00725363293264797
1.88 0.0215122777600197
1.92 0.0340849457147619
1.96 0.0445586559680516
2 0.0526193199340001
2.04 0.0580654344358553
2.08 0.0608306017523714
2.12 0.0609830260077082
2.16 0.0587028315593467
2.2 0.0542554885738363
2.24 0.0479767200382248
2.28 0.0402608270664313
2.32 0.0315428377132281
2.36 0.0222726531970992
2.4 0.0128905505228398
2.44 0.00380872802212489
2.48 -0.00460131440961205
2.52 -0.0120209394238474
2.56 -0.0181957311349107
2.6 -0.0229440186693998
2.64 -0.0261595308611501
2.68 -0.0278099026309769
2.72 -0.0279321847774732
2.76 -0.0266270236741882
2.8 -0.024050438266147
2.84 -0.0204033953198358
2.88 -0.0159194086224311
2.92 -0.0108520244363827
2.96 -0.00546240597382699
3 -7.85887507122509e-06
3.04 0.00526912529787977
3.08 0.0101504743887287
3.12 0.0144505368034553
3.16 0.0180217794996069
3.2 0.0207586106849277
3.24 0.0225988903934062
3.28 0.0235237425515386
3.32 0.0235554237674779
3.36 0.0227535745948298
3.4 0.0212099509858407
3.44 0.0190422090858729
3.48 0.0163869561793244
3.52 0.0133926242216441
3.56 0.0102121424254837
3.6 0.00699603177678566
3.64 0.0038859126790441
3.68 0.0010090284863605
3.72 -0.00152631775583427
3.76 -0.00363414247480506
3.8 -0.00525299073976253
3.84 -0.00634688716849001
3.88 -0.0069052122552593
3.92 -0.00694135856196324
3.96 -0.00649053104792738
4 -0.00560665794392448
};
\addlegendentry{predict}
\addplot [thick, blue, dashed]
table {%
0.04 0.000140014946965442
0.08 0.000720146717751033
0.12 0.00197602267077028
0.16 0.00406630075681433
0.2 0.00717070410748714
0.24 0.0114893991065329
0.28 0.0171729726273926
0.32 0.024257281776499
0.36 0.0326987267658206
0.4 0.0424285762379663
0.44 0.0533561451606472
0.48 0.0653477803735181
0.52 0.0782216355847187
0.56 0.0917665500780245
0.6 0.105768777825507
0.64 0.120020810406467
0.68 0.13431974543689
0.72 0.148471557674555
0.76 0.162302327107032
0.8 0.175668855199718
0.84 0.185522230487431
0.88 0.18828550757443
0.92 0.184970468489367
0.96 0.176949786809227
1 0.163648549526998
1.04 0.144513464918043
1.08 0.120465058402529
1.12 0.0938621424309684
1.16 0.0664515627481474
1.2 0.0389259722493751
1.24 0.0119400428918729
1.28 -0.0133907463580785
1.32 -0.0358070142020983
1.36 -0.0544105602562812
1.4 -0.0688506245285249
1.44 -0.0789778438919373
1.48 -0.084635381533005
1.52 -0.085780796079188
1.56 -0.0826372626027194
1.6 -0.0756914754040303
1.64 -0.0655601174438489
1.68 -0.0528848704188393
1.72 -0.0383460996403683
1.76 -0.0226868952813502
1.8 -0.00668067250678404
1.84 0.00893899235555523
1.88 0.0235251969504032
1.92 0.0365269785765645
1.96 0.0474860865811452
2 0.0560516421229767
2.04 0.0619996785684976
2.08 0.0652405704349508
2.12 0.0658084515552972
2.16 0.063845892791152
2.2 0.0595906964641053
2.24 0.0533642771868648
2.28 0.0455538520527164
2.32 0.0365899952850025
2.36 0.0269230692316715
2.4 0.0170034009417749
2.44 0.00726331748549851
2.48 -0.00189933798223835
2.52 -0.0101368920634583
2.56 -0.0171639299778142
2.6 -0.0227654455213741
2.64 -0.0268005958154968
2.68 -0.0292039639538364
2.72 -0.0299836054833175
2.76 -0.029216308487116
2.8 -0.0270396011135873
2.84 -0.0236419008618425
2.88 -0.0192510248631362
2.92 -0.0141221757337677
2.96 -0.00852523851599646
3 -0.00273240247510174
3.04 0.00299372321993861
3.08 0.00841051249141467
3.12 0.0133044487767333
3.16 0.0174982532606038
3.2 0.0208561842723831
3.24 0.0232869690622002
3.28 0.0247448201422215
3.32 0.0252284369980071
3.36 0.024778225818181
3.4 0.0234717489664883
3.44 0.021418097623957
3.48 0.0187510972575327
3.52 0.01562205134893
3.56 0.0121921386114816
3.6 0.00862492875073749
3.64 0.00507921253767717
3.68 0.00170264485593567
3.72 -0.00137393298059126
3.76 -0.00404094126616981
3.8 -0.00621333745426983
3.84 -0.00783261341346062
3.88 -0.00886760117944158
3.92 -0.00931399360396492
3.96 -0.00919282589177759
4 -0.00854787571613503
};
\addlegendentry{truth}
\addplot [thick, black, dotted, mark=*, mark size=1, mark options={solid}]
table {%
0.04 0.00603631728632963
0.08 -0.0077733965877598
0.12 -0.00735574203750644
0.16 -0.00961473483717689
0.2 -0.00540529462381191
0.24 0.028459615358213
0.28 0.0271095225973211
0.32 0.0343188236358414
0.36 0.0557490785736827
0.4 0.0585238466639939
0.44 0.0557583586579828
0.48 0.0655808807190827
0.52 0.0683667847695446
0.56 0.0887949686025835
0.6 0.0882065646055279
0.64 0.0962734810437313
0.68 0.153211353400852
0.72 0.167229891732418
0.76 0.170229998947223
0.8 0.172197493706253
0.84 0.195691181159531
0.88 0.19863078172906
0.92 0.181760862615536
0.96 0.157082377742558
1 0.14952341584979
1.04 0.141545123071201
1.08 0.0979208654468057
1.12 0.0911592613391374
1.16 0.0700203576660439
1.2 0.0398172888386467
1.24 0.00861767201656329
1.28 -0.00254544149650058
1.32 -0.0472018561482901
1.36 -0.0603892811059559
1.4 -0.0733402764396044
1.44 -0.0821991823486962
1.48 -0.0923266052562572
1.52 -0.0863482126061715
1.56 -0.0927309000896363
1.6 -0.0740343646660173
1.64 -0.0553679615624526
1.68 -0.0492167257573037
1.72 -0.065423370966397
1.76 -0.0279021800830455
1.8 0.00597484786477334
1.84 0.00612327551472035
1.88 0.0346566366478349
1.92 0.014806321408206
1.96 0.0538074193918866
2 0.0512712068974447
};
\addlegendentry{data}
\end{axis}

\end{tikzpicture}
		\caption{Noise level $6\%$}
	\end{subfigure}
	\hfill
	\begin{subfigure}{0.45\textwidth}
\begin{tikzpicture}

\begin{axis}[
axis background/.style={fill=white!93.3333333333333!black},
axis line style={white!73.7254901960784!black},
legend cell align={left},
legend style={
  fill opacity=0.8,
  draw opacity=1,
  text opacity=1,
  draw=white!80!black,
  fill=white!93.3333333333333!black
},
tick pos=left,
width=\textwidth,
x grid style={white!69.8039215686274!black},
xlabel={\(\displaystyle t\)},
xmajorgrids,
xmin=-0.158, xmax=4.198,
xtick style={color=black},
y grid style={white!69.8039215686274!black},
ylabel={Tip displacement},
ymajorgrids,
ymin=-0.127832993423404, ymax=0.215203167063198,
ytick style={color=black}
]
\addplot [thick, white!50.1960784313725!black]
table {%
0.04 0.000118245622784269
0.08 0.00062121067950784
0.12 0.00174373105008721
0.16 0.00366579387251918
0.2 0.00655908769417644
0.24 0.0105487347275054
0.28 0.0156920213840331
0.32 0.0219895634889271
0.36 0.0294014987415853
0.4 0.0378507035787044
0.44 0.047228726080069
0.48 0.0574053104663192
0.52 0.068238402874897
0.56 0.0795819015147969
0.6 0.09129265513684
0.64 0.103236642151981
0.68 0.115294053777973
0.72 0.127362727389473
0.76 0.139360342893173
0.8 0.151225376588065
0.84 0.160433796798694
0.88 0.163733178588812
0.92 0.161517476784078
0.96 0.154714721016651
1 0.143342868415664
1.04 0.128240714155524
1.08 0.11074571269518
1.12 0.0920327695704013
1.16 0.0729942824753509
1.2 0.0544926342540255
1.24 0.0372996008327564
1.28 0.022013847281947
1.32 0.00903946920135941
1.36 -0.001377168026714
1.4 -0.00913919528502781
1.44 -0.0142887695845025
1.48 -0.0169897318144086
1.52 -0.0174979021033599
1.56 -0.0161372629569025
1.6 -0.013274109190398
1.64 -0.0092952263765727
1.68 -0.004584411320193
1.72 0.000494517097781838
1.76 0.0056132914714162
1.8 0.0104892466516876
1.84 0.0148935271901225
1.88 0.0186536957463515
1.92 0.021654875871796
1.96 0.0238364824132723
2 0.0251879219153225
2.04 0.0257415217661726
2.08 0.0255652817145896
2.12 0.0247540218066475
2.16 0.0234212079779176
2.2 0.0216903764543967
2.24 0.019688012008555
2.28 0.0175365259941899
2.32 0.0153490836798642
2.36 0.0132250654137318
2.4 0.0112473351527346
2.44 0.00948033275987298
2.48 0.00796962967209247
2.52 0.00674217253984115
2.56 0.0058077427299394
2.6 0.0051607149881574
2.64 0.00478264404062891
2.68 0.00464481145333461
2.72 0.00471120411661611
2.76 0.00494114935555268
2.8 0.00529210261169478
2.84 0.00572187157402617
2.88 0.00619066105248366
2.92 0.00666258885662663
2.96 0.00710700953563177
3 0.00749915400923759
3.04 0.00782057021816032
3.08 0.00805908659334177
3.12 0.00820856887746814
3.16 0.0082682296006404
3.2 0.00824199048099535
3.24 0.00813764628892795
3.28 0.00796592824947189
3.32 0.00773941075596619
3.36 0.00747177752283447
3.4 0.00717695559251578
3.44 0.00686831397936787
3.48 0.00655816470722358
3.52 0.00625739477413983
3.56 0.00597491648730938
3.6 0.00571761888721662
3.64 0.00549032011954107
3.68 0.00529579931715787
3.72 0.00513490223048219
3.76 0.00500684264906361
3.8 0.00490941904037512
3.84 0.00483933103517193
3.88 0.00479248729765594
3.92 0.00476433706202038
3.96 0.00475013182360721
4 0.00474519323116127
};
\addlegendentry{initial}
\addplot [thick, red]
table {%
0.04 0.000148799464290964
0.08 0.000754310923590242
0.12 0.00204593899930854
0.16 0.00416864998873721
0.2 0.00729147821684869
0.24 0.0116417128028101
0.28 0.0174253721912796
0.32 0.024685821583719
0.36 0.0333323859942084
0.4 0.0432737467577764
0.44 0.0544450706089099
0.48 0.0667466455376869
0.52 0.0799986422995621
0.56 0.0939467968306316
0.6 0.108341592681016
0.64 0.122987021759197
0.68 0.137699969843172
0.72 0.152278473690736
0.76 0.166517026713249
0.8 0.180244885340938
0.84 0.190231985450895
0.88 0.192930279626173
0.92 0.189638130221019
0.96 0.181875319771997
1 0.169007793721586
1.04 0.149708355101188
1.08 0.124296317236353
1.12 0.0960930358991264
1.16 0.0679834546815379
1.2 0.0401967012396269
1.24 0.0124079826240784
1.28 -0.0144437218605323
1.32 -0.0384479661014705
1.36 -0.0577553252668092
1.4 -0.0721226855571644
1.44 -0.0823478418617935
1.48 -0.0884474077401024
1.52 -0.0898117315725183
1.56 -0.0861610232579806
1.6 -0.0780381059789585
1.64 -0.0666824041238157
1.68 -0.0530348999932495
1.72 -0.0374257887981553
1.76 -0.0202763006530843
1.8 -0.002491419539767
1.84 0.0147960728411779
1.88 0.030698531799618
1.92 0.0448096925117412
1.96 0.0569043265423458
2 0.0665550401427423
2.04 0.0732758474582262
2.08 0.0768530941506311
2.12 0.0774199125613419
2.16 0.0752952227035777
2.2 0.0707367772195707
2.24 0.0639434410907503
2.28 0.0552594841789084
2.32 0.0452271951274928
2.36 0.0344476868538137
2.4 0.0234319258548328
2.44 0.0125802228548513
2.48 0.00228153237080365
2.52 -0.00703099010736552
2.56 -0.0149628092844162
2.6 -0.0212574164139775
2.64 -0.0257969678323103
2.68 -0.0285379554009976
2.72 -0.0294633355291945
2.76 -0.0286089364813709
2.8 -0.0261157534501017
2.84 -0.0222225309199075
2.88 -0.0172087527230856
2.92 -0.0113533452947846
2.96 -0.00493674636236043
3 0.00173763880140831
3.04 0.00834975915601158
3.08 0.0146020532677084
3.12 0.0202503618058956
3.16 0.025102262106313
3.2 0.0290021804840345
3.24 0.0318296998136124
3.28 0.0335158749009078
3.32 0.0340564225999805
3.36 0.0335040174075833
3.4 0.0319495507645833
3.44 0.0295118005842874
3.48 0.0263381633026351
3.52 0.0226051501836511
3.56 0.0185077243822864
3.6 0.0142418674761212
3.64 0.00999275238486038
3.68 0.00593225708137125
3.72 0.0022177076964998
3.76 -0.00101438017640992
3.8 -0.00365759589096993
3.84 -0.00564073471882618
3.88 -0.00692611896248591
3.92 -0.00750562893340173
3.96 -0.00739939918296532
4 -0.0066560039928192
};
\addlegendentry{predict}
\addplot [thick, blue, dashed]
table {%
0.04 0.000140014946965442
0.08 0.000720146717751033
0.12 0.00197602267077028
0.16 0.00406630075681433
0.2 0.00717070410748714
0.24 0.0114893991065329
0.28 0.0171729726273926
0.32 0.024257281776499
0.36 0.0326987267658206
0.4 0.0424285762379663
0.44 0.0533561451606472
0.48 0.0653477803735181
0.52 0.0782216355847187
0.56 0.0917665500780245
0.6 0.105768777825507
0.64 0.120020810406467
0.68 0.13431974543689
0.72 0.148471557674555
0.76 0.162302327107032
0.8 0.175668855199718
0.84 0.185522230487431
0.88 0.18828550757443
0.92 0.184970468489367
0.96 0.176949786809227
1 0.163648549526998
1.04 0.144513464918043
1.08 0.120465058402529
1.12 0.0938621424309684
1.16 0.0664515627481474
1.2 0.0389259722493751
1.24 0.0119400428918729
1.28 -0.0133907463580785
1.32 -0.0358070142020983
1.36 -0.0544105602562812
1.4 -0.0688506245285249
1.44 -0.0789778438919373
1.48 -0.084635381533005
1.52 -0.085780796079188
1.56 -0.0826372626027194
1.6 -0.0756914754040303
1.64 -0.0655601174438489
1.68 -0.0528848704188393
1.72 -0.0383460996403683
1.76 -0.0226868952813502
1.8 -0.00668067250678404
1.84 0.00893899235555523
1.88 0.0235251969504032
1.92 0.0365269785765645
1.96 0.0474860865811452
2 0.0560516421229767
2.04 0.0619996785684976
2.08 0.0652405704349508
2.12 0.0658084515552972
2.16 0.063845892791152
2.2 0.0595906964641053
2.24 0.0533642771868648
2.28 0.0455538520527164
2.32 0.0365899952850025
2.36 0.0269230692316715
2.4 0.0170034009417749
2.44 0.00726331748549851
2.48 -0.00189933798223835
2.52 -0.0101368920634583
2.56 -0.0171639299778142
2.6 -0.0227654455213741
2.64 -0.0268005958154968
2.68 -0.0292039639538364
2.72 -0.0299836054833175
2.76 -0.029216308487116
2.8 -0.0270396011135873
2.84 -0.0236419008618425
2.88 -0.0192510248631362
2.92 -0.0141221757337677
2.96 -0.00852523851599646
3 -0.00273240247510174
3.04 0.00299372321993861
3.08 0.00841051249141467
3.12 0.0133044487767333
3.16 0.0174982532606038
3.2 0.0208561842723831
3.24 0.0232869690622002
3.28 0.0247448201422215
3.32 0.0252284369980071
3.36 0.024778225818181
3.4 0.0234717489664883
3.44 0.021418097623957
3.48 0.0187510972575327
3.52 0.01562205134893
3.56 0.0121921386114816
3.6 0.00862492875073749
3.64 0.00507921253767717
3.68 0.00170264485593567
3.72 -0.00137393298059126
3.76 -0.00404094126616981
3.8 -0.00621333745426983
3.84 -0.00783261341346062
3.88 -0.00886760117944158
3.92 -0.00931399360396492
3.96 -0.00919282589177759
4 -0.00854787571613503
};
\addlegendentry{truth}
\addplot [thick, black, dotted, mark=*, mark size=1, mark options={solid}]
table {%
0.04 0.0252924277280806
0.08 -0.00526302176699334
0.12 -0.0156487424417063
0.16 0.00757891850459556
0.2 0.0181443395796639
0.24 0.00425154103382267
0.28 0.010243351859313
0.32 0.0635470499797471
0.36 0.0265001452365009
0.4 0.0277605734671877
0.44 0.055775529456014
0.48 0.0723563709961322
0.52 0.0968833202697006
0.56 0.0851506488818199
0.6 0.101572855754841
0.64 0.139624489540715
0.68 0.126328253050826
0.72 0.164520636717078
0.76 0.186756401884119
0.8 0.178400608370176
0.84 0.18303521551628
0.88 0.199610614313807
0.92 0.176708777694456
0.96 0.167589931352379
1 0.170599051178352
1.04 0.158972506346563
1.08 0.115859195675929
1.12 0.0945865657044354
1.16 0.0671190025766413
1.2 0.0483044163463201
1.24 0.0150405575784056
1.28 -0.020498289629689
1.32 -0.0344262532031038
1.36 -0.0623919893800088
1.4 -0.069339413762601
1.44 -0.0678202980363061
1.48 -0.0985986122221554
1.52 -0.112240440674013
1.56 -0.0578900524966399
1.6 -0.0718675111200538
1.64 -0.0855833352503301
1.68 -0.0594133621292621
1.72 -0.0147578514606414
1.76 -0.0351198211436823
1.8 0.00457454036256481
1.84 -0.00336868496862355
1.88 0.0402723227541884
1.92 0.0496810472727153
1.96 0.0628929360774439
2 0.0596831077744781
};
\addlegendentry{data}
\end{axis}

\end{tikzpicture}
		\caption{Noise level $8\%$}
	\end{subfigure}
	
	\caption{Evolution of the tip displacement comparing the ``true" model and the model calibrated from the noisy data with the noise level $2,4,6,8\%$.}
	\label{fig:tip_displacement}
\end{figure}
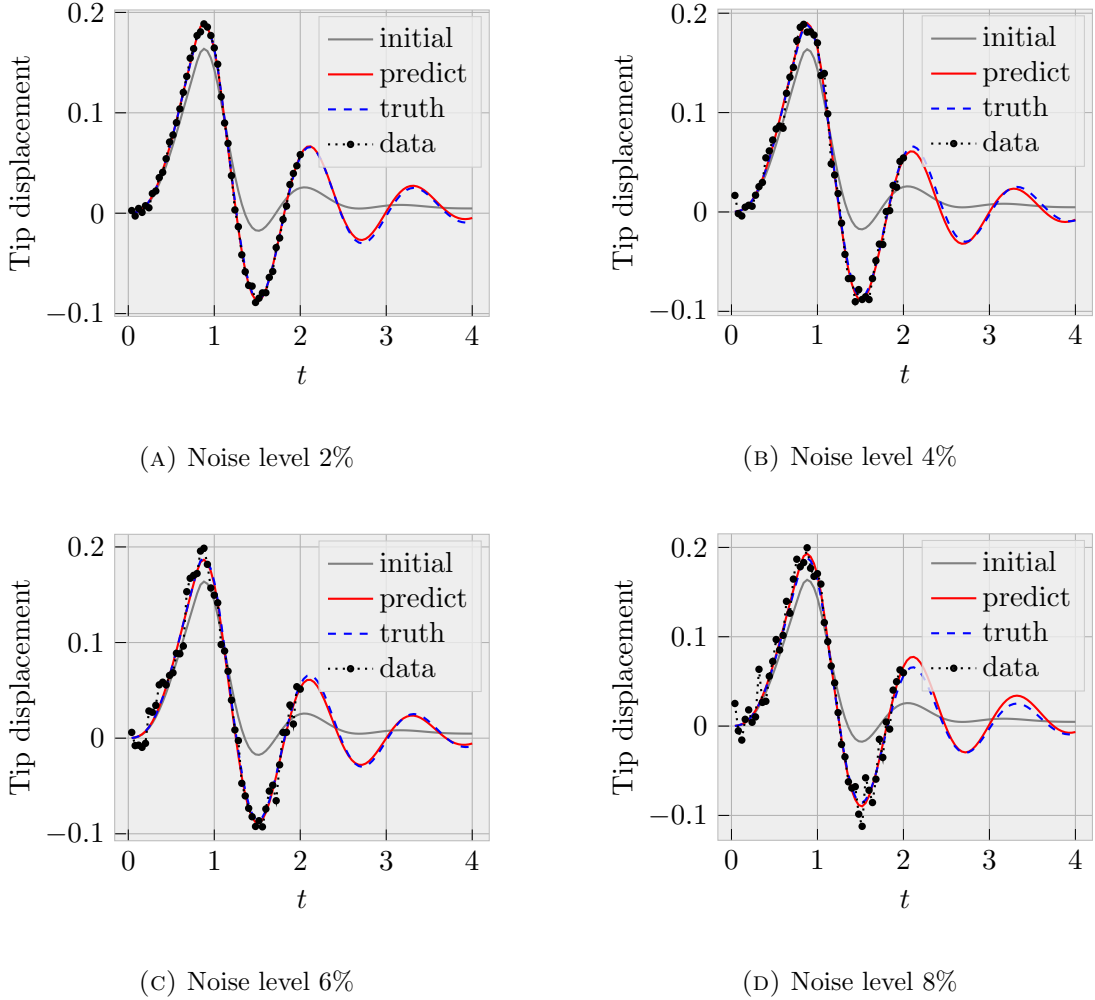

\begin{figure}[!h]
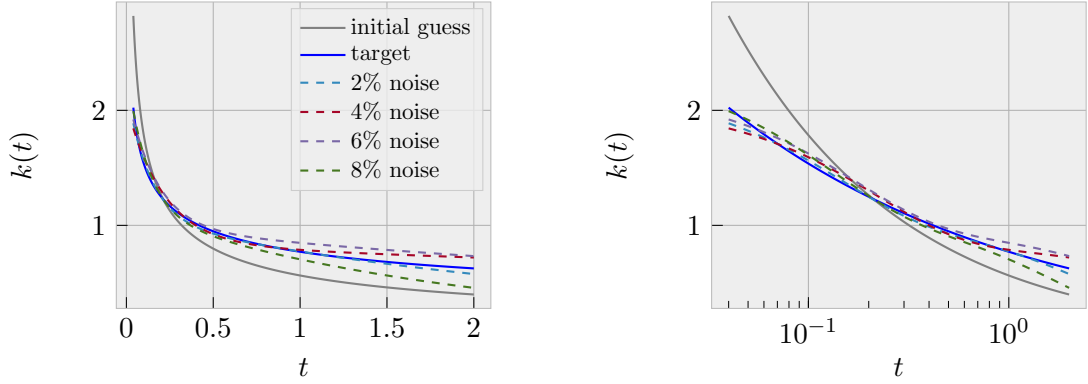

	\input{plt_compare_resulting_kernels.tex}
	\hfill
	\input{plt_compare_resulting_kernels_log.tex}
	\caption{Comparison of the resulting kernels calibrated using different noise level (right: log-scale).}
	\label{fig:kernels}
\end{figure}

\begin{figure}[!h]
\begin{tikzpicture}

\definecolor{color0}{rgb}{0.203921568627451,0.541176470588235,0.741176470588235}
\definecolor{color1}{rgb}{0.650980392156863,0.0235294117647059,0.156862745098039}
\definecolor{color2}{rgb}{0.47843137254902,0.407843137254902,0.650980392156863}
\definecolor{color3}{rgb}{0.274509803921569,0.470588235294118,0.129411764705882}

\begin{axis}[
axis background/.style={fill=white!93.3333333333333!black},
axis line style={white!73.7254901960784!black},
legend cell align={left},
legend style={
  fill opacity=0.8,
  draw opacity=1,
  text opacity=1,
  draw=white!80!black,
  fill=white!93.3333333333333!black
},
log basis y={10},
tick pos=left,
width=0.6\textwidth,
x grid style={white!69.8039215686274!black},
xlabel={\(\displaystyle iteration\)},
xmajorgrids,
xmin=0.35, xmax=14.65,
xtick style={color=black},
y grid style={white!69.8039215686274!black},
ylabel={Loss},
ymajorgrids,
ymin=0.000251832935434445, ymax=0.345527640898305,
ymode=log,
ytick style={color=black}
]
\addplot [thick, color0, mark=*, mark size=1, mark options={solid}]
table {%
1 0.233947804866796
2 0.0149157959690183
3 0.00309658581742947
4 0.00195501292722077
5 0.00126875405309699
6 0.000616705459490229
7 0.000401082537639969
8 0.000397553861970357
9 0.000395764614107594
10 0.000387873438317473
11 0.000375571787281794
12 0.000359289391075803
13 0.000350013778477817
14 0.000349719836910302
};
\addlegendentry{$2\%$ noise}
\addplot [thick, color1, mark=*, mark size=1, mark options={solid}]
table {%
1 0.240886880121253
2 0.0176945480573462
3 0.00481172780221422
4 0.00353382180448269
5 0.00268611422694161
6 0.00189912413005166
7 0.00161824289337072
8 0.00161369831751755
9 0.00161119416483475
10 0.00160021530437584
11 0.00158320055598934
12 0.00156115431128344
13 0.00154936749120421
14 0.00154892388832997
};
\addlegendentry{$4\%$ noise}
\addplot [thick, color2, mark=*, mark size=1, mark options={solid}]
table {%
1 0.248676048931309
2 0.021179992870255
3 0.00681156982325201
4 0.0054406578403433
5 0.0045615122569166
6 0.0037399084941407
7 0.00343335012129681
8 0.00343044899782337
9 0.00343032114729377
10 0.00342979924469594
11 0.00342637103374873
12 0.00342476825776582
13 0.00342440962825711
14 0.00342440378508464
};
\addlegendentry{$6\%$ noise}
\addplot [thick, color3, mark=*, mark size=1, mark options={solid}]
table {%
1 0.248814138911648
2 0.0218331096870561
3 0.00851141910682748
4 0.00719869638533542
5 0.00623375448453886
6 0.00530841631554953
7 0.00487367080331783
8 0.00483570861526034
9 0.00479248312832013
10 0.00466592178975209
11 0.00448267413784649
12 0.00433838333554307
13 0.0043046146242061
14 0.00430258619424591
};
\addlegendentry{$8\%$ noise}
\end{axis}

\end{tikzpicture}
\caption{Convergence of the loss function during the optimization.}
\label{fig:loss_convergence}
\end{figure}
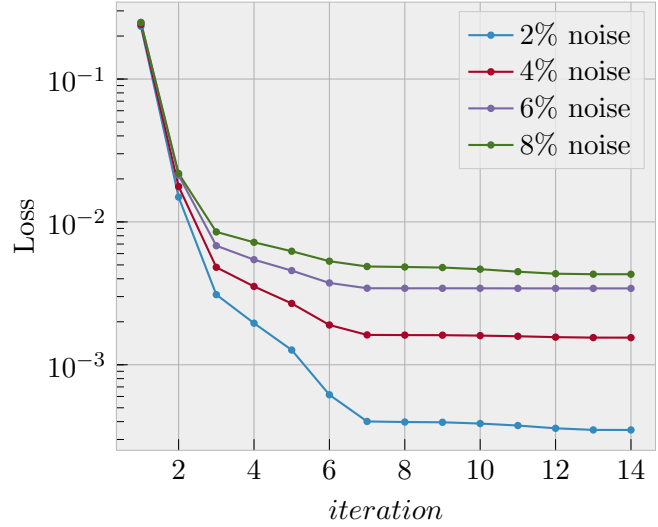

\begin{figure}[!h]
	\input{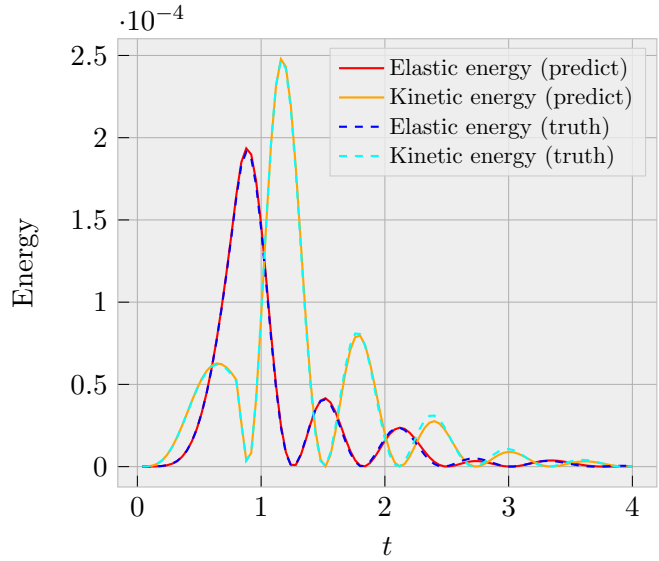}
	\caption{Evolution of the kinetic and elastic energies for the "true" and calibrated models (noise level $2\%$).}
	\label{fig:energy}
\end{figure}

Next, we consider the case when $\kertreps^{true}(t)$ is given by a $22$ modes rational approximation of $t^{\alpha_1-1}/\Gamma(\alpha_1)$ with $\alpha_1=0.9$, and $\kereps^{true}(t)$ is given by a $22$ modes rational approximation of $t^{\alpha_2-1}/\Gamma(\alpha_2)$ with $\alpha_2=0.7$.
We use a $8$ modes rational approximation of $t^{\alpha_0-1}/\Gamma(\alpha_0)$ with $\alpha_0=0.5$ as initial guess for both kernels.
Using this kernel, we numerically generate synthetic measurements of the tip displacement on the interval $[0,t_{meas}]$, using both loading types, bending~\eqref{eq:load:bending} and extension~\eqref{eq:load:extension}, and adding a noise of level $2\%$.
Then, we predict the kernel $\kertreps^{pred}(t)$ and $\kereps^{pred}(t)$ from these measurements with the number of modes $m_{\treps}=m_{\epsi}=8$.
To do this, we infer the parameters $\theta = \{w_{\treps,k},\lambda_{\treps,k},w_{\epsi,k}, \lambda_{\epsi,k}\}$ minimizing the loss function $J(\theta)$ defined as
\begin{equation}\label{key}
J(\theta) = \frac{1}{2}\frac{\|u_{tip}^{bend}(\theta) - u_{tip}^{bend,meas}\|^2_{[0,t_{meas}]}}{\|u_{tip}^{bend,meas}\|^2_{[0,t_{meas}]}} + \omega\cdot\frac{1}{2}\frac{\|u_{tip}^{ext}(\theta) - u_{tip}^{ext,meas}\|^2_{[0,t_{meas}]}}{\|u_{tip}^{ext,meas}\|^2_{[0,t_{meas}]}},
\end{equation}
where the superscipts $bend$ and $ext$ correspond to the bending and extension excitation types, respectively; the norm is again the $\ell_2$-norm on the discretized interval $[0,t_{meas}]$.
Above, the weight $\omega$ is introduced to balance the contributions due to the bending and the extension terms.
In our experiment, we fix $\omega=10$ in order to take into account the tip displacement measurements under extension after the time $t_{meas}$, which have a scale of an order of magnitude smaller than the global scale $\|u_{tip}^{ext,meas}\|_{[0,t_{meas}]}$ (see \Cref{sbfig:two_kernels:observations:extension}).
In~\Cref{fig:two_kernels:observations}, the norm of the resulting tip displacements is plotted on the complete interval $[0,T]$ for the model calibrated from the noisy data.
Again, we can see that the calibrated model can predict the evolution after the time $t_{meas}$.
In~\Cref{fig:two_kernels:kernels}, where the resulting kernels are compared with the target, we can observe that providing measurements for different loading types results in a good approximation of the integral kernels.

\begin{figure}[!h]
	\begin{subfigure}{0.45\textwidth}
\begin{tikzpicture}

\begin{axis}[
axis background/.style={fill=white!93.3333333333333!black},
axis line style={white!73.7254901960784!black},
legend cell align={left},
legend style={
  fill opacity=0.8,
  draw opacity=1,
  text opacity=1,
  draw=white!80!black,
  fill=white!93.3333333333333!black
},
yticklabel style={
	/pgf/number format/fixed,
	/pgf/number format/precision=5
},
tick pos=left,
width=\textwidth,
x grid style={white!69.8039215686274!black},
xlabel={\(\displaystyle t\)},
xmajorgrids,
xmin=-0.158, xmax=4.198,
xtick style={color=black},
y grid style={white!69.8039215686274!black},
ylabel={Tip displacement norm},
ymajorgrids,
ymin=-0.00960925077972064, ymax=0.204404959270626,
ytick style={color=black}
]
\addplot [thick, white!50.1960784313725!black]
table {%
0.04 0.000118667858931464
0.08 0.000622860025974365
0.12 0.0017472446403227
0.16 0.00367176262765436
0.2 0.00656829281869306
0.24 0.0105621204598141
0.28 0.0157105835400944
0.32 0.0220142877057838
0.36 0.0294333648620439
0.4 0.0378906680129879
0.44 0.0472777020866007
0.48 0.0574641582500944
0.52 0.0683079348446382
0.56 0.0796628917195443
0.6 0.0913858541231983
0.64 0.103342796024496
0.68 0.11541392727814
0.72 0.127497126047365
0.76 0.139510134776702
0.8 0.151391510751446
0.84 0.160610787199179
0.88 0.163914123962149
0.92 0.161704471821617
0.96 0.154911798208426
1 0.143553120932653
1.04 0.128470426608106
1.08 0.111007051066813
1.12 0.0923456098306162
1.16 0.0733909576090769
1.2 0.0550312482262103
1.24 0.0380982641839804
1.28 0.023367187701343
1.32 0.0119972491397656
1.36 0.00801277947960942
1.4 0.0120395640473088
1.44 0.0162363325015787
1.48 0.0185752487652508
1.52 0.0189354520723284
1.56 0.0175509476272621
1.6 0.0147853255989687
1.64 0.0111079771026344
1.68 0.00725498460655859
1.72 0.00517176273452833
1.76 0.00730044383773307
1.8 0.0112955510666202
1.84 0.0153524476861084
1.88 0.0189393791260936
1.92 0.0218416060542508
1.96 0.0239617993811178
2 0.0252730798871356
2.04 0.0257996373018923
2.08 0.0256051114827685
2.12 0.0247819339481552
2.16 0.0234423394008766
2.2 0.0217093942919718
2.24 0.0197097018644725
2.28 0.0175663085334104
2.32 0.0153935292274598
2.36 0.0132924104120308
2.4 0.0113479298715057
2.44 0.00962677458478534
2.48 0.00817608523950625
2.52 0.00702214040121231
2.56 0.00616979290748406
2.6 0.00560324681849907
2.64 0.00529084766906318
2.68 0.00519302646964199
2.72 0.00527000757203051
2.76 0.00548482230268598
2.8 0.00580272174005344
2.84 0.00618958321375258
2.88 0.00661177651881067
2.92 0.00703729584049296
2.96 0.00743755539399218
3 0.00778864873587582
3.04 0.0080723140478662
3.08 0.008276245351656
3.12 0.00839404788560903
3.16 0.00842466896754525
3.2 0.00837181516582236
3.24 0.00824314002631297
3.28 0.00804932572265528
3.32 0.00780300032265828
3.36 0.00751799619427461
3.4 0.00720847480697677
3.44 0.00688810829831269
3.48 0.00656954982383803
3.52 0.00626402348130142
3.56 0.00598071297308442
3.6 0.00572664819255013
3.64 0.00550659520914044
3.68 0.00532304697967301
3.72 0.00517631879400457
3.76 0.00506488733271502
3.8 0.00498567736192182
3.84 0.0049344686223697
3.88 0.00490629703789584
3.92 0.00489585735630358
3.96 0.00489780639806885
4 0.00490704408031812
};
\addlegendentry{initial}
\addplot [thick, red]
table {%
0.04 0.000144160377910242
0.08 0.000737073930411183
0.12 0.00201291191227981
0.16 0.00412539102565453
0.2 0.00724985328714739
0.24 0.0115946055479997
0.28 0.0173298991624474
0.32 0.0245004072013151
0.36 0.0330550455785748
0.4 0.0429198229177312
0.44 0.0540098106329507
0.48 0.0662002365502575
0.52 0.0793098982201568
0.56 0.0931159926539638
0.6 0.107391320225749
0.64 0.121921100915514
0.68 0.136496057537639
0.72 0.150910213966001
0.76 0.164972088863615
0.8 0.178520895581219
0.84 0.188412857893744
0.88 0.19104362538593
0.92 0.187528522676743
0.96 0.179311198831686
1 0.165807682698905
1.04 0.14620909037349
1.08 0.121194630359083
1.12 0.0933455179841484
1.16 0.0647769554190442
1.2 0.036274024198681
1.24 0.00984739364263562
1.28 0.020937824779935
1.32 0.0445357311194831
1.36 0.0645512760472424
1.4 0.0802181618491865
1.44 0.0914258296912398
1.48 0.0979918618802475
1.52 0.0997292556770743
1.56 0.0967165273029786
1.6 0.0894101577800468
1.64 0.0784983104584466
1.68 0.0646686816567712
1.72 0.0485914132264326
1.76 0.0310229727240262
1.8 0.0128873684872585
1.84 0.00595046947857627
1.88 0.0227152334854242
1.92 0.0382484549656614
1.96 0.0517128354212435
2 0.0626346090058042
2.04 0.0706762022917826
2.08 0.075657721931001
2.12 0.077554479389637
2.16 0.076469994794083
2.2 0.0726117703816864
2.24 0.0662853110438
2.28 0.0578877837394909
2.32 0.0478885933272061
2.36 0.0368001913798885
2.4 0.0251584829227612
2.44 0.0135525064289002
2.48 0.00364361449633154
2.52 0.00926164600311809
2.56 0.0182540738241512
2.6 0.0258956253577862
2.64 0.0318072652502801
2.68 0.0358203371796703
2.72 0.0378710632899007
2.76 0.0379856129232915
2.8 0.0362716885024391
2.84 0.0329086768699231
2.88 0.0281358171918834
2.92 0.022241468385337
2.96 0.0155591498387358
3 0.00851434486436527
3.04 0.00280128324474619
3.08 0.00715759900839069
3.12 0.0137119749352839
3.16 0.01974234950487
3.2 0.0248977564266757
3.24 0.0289863049562011
3.28 0.0318886103566315
3.32 0.0335461793591543
3.36 0.0339573166189
3.4 0.033172857608097
3.44 0.0312908378990627
3.48 0.028449710816621
3.52 0.0248206679158456
3.56 0.0205993642868426
3.6 0.0159985337103607
3.64 0.0112451640022172
3.68 0.00660427258203618
3.72 0.00270719049975335
3.76 0.00326884762343487
3.8 0.00648139607796663
3.84 0.00934705106371928
3.88 0.0115382226730013
3.92 0.0129671400885108
3.96 0.0136106059275903
4 0.0134876528049059
};
\addlegendentry{predict}
\addplot [thick, blue, dashed]
table {%
0.04 0.000144444835800665
0.08 0.000738822363342573
0.12 0.00201819060055682
0.16 0.00413611643973659
0.2 0.00726779902065134
0.24 0.0116235170967876
0.28 0.017377036578403
0.32 0.0245764740211434
0.36 0.0331731350093548
0.4 0.0430935672944019
0.44 0.0542527012110571
0.48 0.0665267421146665
0.52 0.0797353682170086
0.56 0.0936552938888629
0.6 0.108057699365018
0.64 0.122725377401157
0.68 0.137447113198305
0.72 0.152015534000765
0.76 0.166237327113245
0.8 0.179949528670914
0.84 0.190000634312257
0.88 0.192771797780044
0.92 0.189377271824929
0.96 0.181281698875223
1 0.16790180313874
1.04 0.148386215663012
1.08 0.123383343122237
1.12 0.0954801813573837
1.16 0.0668282735602639
1.2 0.0383059101776933
1.24 0.0122336412670915
1.28 0.0203937866360874
1.32 0.0437648106877269
1.36 0.0638838725102692
1.4 0.0796891488437472
1.44 0.0910118002053593
1.48 0.0976684434691038
1.52 0.0994784701036828
1.56 0.096509785959004
1.6 0.0892112903989766
1.64 0.0782757024952186
1.68 0.0644102268570841
1.72 0.0483141302920467
1.76 0.0307886410042255
1.8 0.0129966320420236
1.84 0.00762927823873302
1.88 0.023609288163918
1.92 0.0389399017768798
1.96 0.0522292271958632
2 0.0629474125702792
2.04 0.0707436211166362
2.08 0.075434082026753
2.12 0.0769966857058974
2.16 0.0755435363442654
2.2 0.0712951921710024
2.24 0.0645716829963974
2.28 0.0557849590803972
2.32 0.0454205943860279
2.36 0.0340112179018759
2.4 0.022126362878928
2.44 0.0104918499159765
2.48 0.00417183601364298
2.52 0.0129223079811656
2.56 0.0219700500023525
2.6 0.0295440681109813
2.64 0.035313215289017
2.68 0.0391186265007885
2.72 0.0409068334296658
2.76 0.0407167679793958
2.8 0.0386712101841611
2.84 0.0349667331107126
2.88 0.0298621202803876
2.92 0.0236695426925716
2.96 0.0167589489604548
3 0.00964901840277478
3.04 0.00424224441216384
3.08 0.00750681455723117
3.12 0.0138030216917431
3.16 0.0197225773117332
3.2 0.0247519908752089
3.24 0.0286687195330763
3.28 0.0313486830853318
3.32 0.0327378003559924
3.36 0.0328433400126728
3.4 0.0317277529338934
3.44 0.0295021448969588
3.48 0.026318597827188
3.52 0.0223616777141746
3.56 0.0178396134104648
3.6 0.0129769348313192
3.64 0.00801596307040152
3.68 0.00331541315365319
3.72 0.0023336653264147
3.76 0.00613635115462508
3.8 0.00965883988731452
3.84 0.0125107180895637
3.88 0.0145848287708013
3.92 0.0158344711480975
3.96 0.0162558891014418
4 0.0158834906795556
};
\addlegendentry{truth}
\addplot [thick, black, dotted, mark=*, mark size=0.65, mark options={solid}]
table {%
0.04 0.000143966172801475
0.08 0.00075056765359186
0.12 0.00208329894205927
0.16 0.00408836064121734
0.2 0.00719916202304883
0.24 0.0112402075646856
0.28 0.0170178533154385
0.32 0.0249619528802173
0.36 0.0326719227576568
0.4 0.0446258420973212
0.44 0.0538590012246016
0.48 0.0672587827720044
0.52 0.0796986852472375
0.56 0.0923808255383659
0.6 0.107293425810152
0.64 0.123439108162103
0.68 0.139695445400457
0.72 0.152081963131282
0.76 0.16498762703953
0.8 0.173131627462688
0.84 0.194677040631974
0.88 0.190247966083269
0.92 0.193427778459185
0.96 0.177360862624438
1 0.170058324358867
1.04 0.145142692398442
1.08 0.120438663741996
1.12 0.0957624365351636
1.16 0.0672856368877562
1.2 0.0385907285133187
1.24 0.0124504001504458
1.28 0.0197777654559184
1.32 0.0450429226290206
1.36 0.0665783228021029
1.4 0.0781015635534126
1.44 0.0907876090938099
1.48 0.0958723212430218
1.52 0.102842655578786
1.56 0.100727472072364
1.6 0.0882703972914662
1.64 0.0795949568681398
1.68 0.0649839758584631
1.72 0.0476494758764579
1.76 0.0303572370992223
1.8 0.0128068953654259
1.84 0.00771463077986478
1.88 0.0244895948095346
1.92 0.0386521970110428
1.96 0.0522181824838367
2 0.0605027903457193
};
\addlegendentry{data}
\end{axis}

\end{tikzpicture}
		\caption{bending}
		\label{sbfig:two_kernels:observations:bending}
	\end{subfigure}
	\hfill
	\begin{subfigure}{0.45\textwidth}
\begin{tikzpicture}

\begin{axis}[
axis background/.style={fill=white!93.3333333333333!black},
axis line style={white!73.7254901960784!black},
legend cell align={left},
legend style={
  fill opacity=0.8,
  draw opacity=1,
  text opacity=1,
  draw=white!80!black,
  fill=white!93.3333333333333!black
},
tick pos=left,
width=\textwidth,
x grid style={white!69.8039215686274!black},
xlabel={\(\displaystyle t\)},
xmajorgrids,
xmin=-0.158, xmax=4.198,
xtick style={color=black},
y grid style={white!69.8039215686274!black},
ylabel={Tip displacement norm},
ymajorgrids,
ymin=-0.00174380353691742, ymax=0.0479777081145792,
ytick style={color=black}
]
\addplot [thick, white!50.1960784313725!black]
table {%
0.04 0.000516265174514247
0.08 0.00173653312673543
0.12 0.003138955820829
0.16 0.00470086438308404
0.2 0.00644012172336861
0.24 0.00827645206486772
0.28 0.0102318781688017
0.32 0.0122716870558399
0.36 0.0143966619029326
0.4 0.0165940827186224
0.44 0.0188590077537226
0.48 0.0211853308738671
0.52 0.0235684035029347
0.56 0.0260040900125034
0.6 0.028488865561659
0.64 0.031019509023038
0.68 0.0335933514232231
0.72 0.0362077879576619
0.76 0.0388607228351279
0.8 0.0415500234024958
0.84 0.0334383745247433
0.88 0.020918448468218
0.92 0.0186822368444251
0.96 0.016773971008825
1 0.0143657540650386
1.04 0.0134772458299619
1.08 0.0120611078718365
1.12 0.0112712685761036
1.16 0.0104084191566212
1.2 0.00975065365653837
1.24 0.0091477741672641
1.28 0.0086257227878141
1.32 0.00815785968861255
1.36 0.00773611066080812
1.4 0.00735084263906046
1.44 0.00699878830529953
1.48 0.00667159546966377
1.52 0.00637015973290985
1.56 0.00608659255036644
1.6 0.00582329475637795
1.64 0.00557354015426037
1.68 0.00534043584328373
1.72 0.00511853387876828
1.76 0.00491123690362289
1.8 0.00471396502007817
1.84 0.00453003959755853
1.88 0.00435539331443137
1.92 0.00419311474457894
1.96 0.00403951426418851
2 0.00389742910841192
2.04 0.00376347523158334
2.08 0.003640207643716
2.12 0.00352446023513525
2.16 0.00341846123474267
2.2 0.00331919298149305
2.24 0.00322854237283223
2.28 0.00314361488460257
2.32 0.0030659182495312
2.36 0.00299270415843542
2.4 0.00292528612730949
2.44 0.00286107716409793
2.48 0.00280127911088064
2.52 0.00274354074652871
2.56 0.00268895712744197
2.6 0.00263545785145551
2.64 0.00258408289731774
2.68 0.00253304849214814
2.72 0.00248339869233848
2.76 0.0024336288001343
2.8 0.00238480408595828
2.84 0.00233566790372065
2.88 0.00228729962522097
2.92 0.00223865142095215
2.96 0.00219080487978271
3 0.00214288404767867
3.04 0.00209595845981375
3.08 0.00204928642631073
3.12 0.0020039058822112
3.16 0.00195917007745704
3.2 0.00191606190260639
3.24 0.00187399130125302
3.28 0.00183386228279996
3.32 0.00179510731409525
3.36 0.00175850559638157
3.4 0.00172349099935222
3.44 0.00169080599524386
3.48 0.00165987112710767
3.52 0.00163124721294728
3.56 0.00160432664547069
3.6 0.0015795585116248
3.64 0.00155630383709487
3.68 0.00153492129319915
3.72 0.00151476262306223
3.76 0.00149611411214372
3.8 0.00147833931848056
3.84 0.00146167407328747
3.88 0.00144551398925244
3.92 0.00143006953662227
3.96 0.00141478633263613
4 0.00139987036009466
};
\addlegendentry{initial}
\addplot [thick, red]
table {%
0.04 0.000999019817610712
0.08 0.00310937266218722
0.12 0.00492805296424362
0.16 0.00676064584798612
0.2 0.00899761975625768
0.24 0.0110841703011232
0.28 0.0132254313630939
0.32 0.015546866787846
0.36 0.0177913268596958
0.4 0.0200969949718639
0.44 0.0224838979762614
0.48 0.0248386002045292
0.52 0.0272466384075657
0.56 0.0296891586357317
0.6 0.0321299396410354
0.64 0.0346058361755575
0.68 0.0371036128392196
0.72 0.039608639409224
0.76 0.0421402549545648
0.8 0.0446866934328881
0.84 0.0262703411603701
0.88 0.0045689087545594
0.92 0.0111472671063513
0.96 0.0116511661644435
1 0.00406063699185172
1.04 0.00750795957638812
1.08 0.00689911382765259
1.12 0.00366375703189305
1.16 0.00558919509874695
1.2 0.00473946123017928
1.24 0.00344332448867241
1.28 0.00442732229508997
1.32 0.00368270107666631
1.36 0.00329109041457243
1.4 0.00360790575577055
1.44 0.00317428155248371
1.48 0.00298587593299103
1.52 0.00308756416135387
1.56 0.00279605179275149
1.6 0.00273757636008902
1.64 0.00272800419902038
1.68 0.00257360335313316
1.72 0.00253682278362342
1.76 0.00250499061886385
1.8 0.00240965021079516
1.84 0.00238424179774574
1.88 0.00234196601607708
1.92 0.00228554293349155
1.96 0.00225730934914349
2 0.00222152795918913
2.04 0.00217806016837588
2.08 0.00215189188767315
2.12 0.00211535639865279
2.16 0.00208333416801345
2.2 0.00205537347270341
2.24 0.00202819537188819
2.28 0.00200113538138675
2.32 0.00198011989295531
2.36 0.00195542699623402
2.4 0.00193417720037324
2.44 0.00191104981331036
2.48 0.00188861458081319
2.52 0.00186456153017643
2.56 0.00184217761421287
2.6 0.0018176122883154
2.64 0.00179530253456732
2.68 0.00177175719757317
2.72 0.00174988948287887
2.76 0.00172738283035012
2.8 0.00170676466825806
2.84 0.00168568336881807
2.88 0.0016668105691561
2.92 0.0016478263761189
2.96 0.0016309654424973
3 0.00161415028482389
3.04 0.00159915064392355
3.08 0.00158398742715065
3.12 0.00157034191836942
3.16 0.00155641110310511
3.2 0.00154380909497502
3.24 0.00153097775141306
3.28 0.00151940064180037
3.32 0.00150766166658874
3.36 0.0014970852435171
3.4 0.0014862064661782
3.44 0.00147634701506776
3.48 0.00146624772942235
3.52 0.00145666525394076
3.56 0.00144669236014094
3.6 0.00143724598597884
3.64 0.0014269359232464
3.68 0.00141687420688529
3.72 0.00140593360071308
3.76 0.00139483425492008
3.8 0.0013827630675392
3.84 0.00137048807889653
3.88 0.00135719534119396
3.92 0.00134373941906497
3.96 0.00132940167112302
4 0.00131499987170842
};
\addlegendentry{predict}
\addplot [thick, blue, dashed]
table {%
0.04 0.0010047493771158
0.08 0.0031315507384938
0.12 0.0049828967270109
0.16 0.00684699653541138
0.2 0.00910330116363208
0.24 0.0112135013628881
0.28 0.013384598332968
0.32 0.0157290350486864
0.36 0.0180007732834553
0.4 0.0203429088726389
0.44 0.0227599722538218
0.48 0.0251529533913387
0.52 0.0276037096476688
0.56 0.0300872341712772
0.6 0.0325755131859377
0.64 0.0351013853964452
0.68 0.0376497829970592
0.72 0.0402097167011075
0.76 0.0427983841431437
0.8 0.045402733334681
0.84 0.0269335449416946
0.88 0.00514923349796567
0.92 0.0113926215551457
0.96 0.0119666254566553
1 0.00485031130216103
1.04 0.00821691813470738
1.08 0.0075807091559474
1.12 0.00473660008577301
1.16 0.00648929466040814
1.2 0.00554653133710463
1.24 0.00450672994466157
1.28 0.00528216789159829
1.32 0.00457269664413051
1.36 0.00431285978674019
1.4 0.0045272046312113
1.44 0.00415546242983876
1.48 0.00401459843503681
1.52 0.00402295020234068
1.56 0.00371877707133892
1.6 0.00360902597860626
1.64 0.00348531667373153
1.68 0.00327778392509632
1.72 0.00315745154742815
1.76 0.00302690871254407
1.8 0.00286115746265486
1.84 0.00274484721847634
1.88 0.00260626602467902
1.92 0.00246242346275837
1.96 0.00234240210704362
2 0.00221913691662053
2.04 0.00210507953185303
2.08 0.00201431774265838
2.12 0.00192731369292876
2.16 0.00185794615012452
2.2 0.00180040802801925
2.24 0.00175412561572069
2.28 0.00171748762312175
2.32 0.0016971631170998
2.36 0.00168425259959902
2.4 0.00168671071891668
2.44 0.00169505249299484
2.48 0.00171224152207208
2.52 0.00173071059192786
2.56 0.0017526246097632
2.6 0.0017713955979315
2.64 0.00178991554504415
2.68 0.00180318317500044
2.72 0.00181353579038187
2.76 0.00181715320290555
2.8 0.00181552070036716
2.84 0.00180549816613585
2.88 0.00178862779254751
2.92 0.00176272250144214
2.96 0.00172980506504986
3 0.00168924490673617
3.04 0.0016432767692013
3.08 0.00159143476380032
3.12 0.00153612751767507
3.16 0.00147681023676043
3.2 0.00141559454810073
3.24 0.00135202663236199
3.28 0.00128814232453725
3.32 0.00122385034005605
3.36 0.00116208781601583
3.4 0.00110283608969331
3.44 0.00104884672531748
3.48 0.00100133918595868
3.52 0.000962973747482536
3.56 0.000934397553601966
3.6 0.000917742953875004
3.64 0.000911769799095603
3.68 0.000917881590807423
3.72 0.000933810766129906
3.76 0.00095839738814227
3.8 0.000989185019182988
3.84 0.00102479735160954
3.88 0.00106227291729605
3.92 0.00110066386805048
3.96 0.00113765046835076
4 0.00117259517588086
};
\addlegendentry{truth}
\addplot [thick, black, dotted, mark=*, mark size=0.65, mark options={solid}]
table {%
0.04 0.00101094575772956
0.08 0.00309542571548585
0.12 0.00503349007708109
0.16 0.00665043833596454
0.2 0.00886565940383157
0.24 0.0112865388183302
0.28 0.0135980816785996
0.32 0.0150763221660988
0.36 0.0178992326114846
0.4 0.0194965879523042
0.44 0.0229956486117268
0.48 0.0245656939115532
0.52 0.0271020656043883
0.56 0.0297556061530199
0.6 0.0325965191523237
0.64 0.0358156762848778
0.68 0.0378102732341072
0.72 0.0399344335170473
0.76 0.0438355554562395
0.8 0.0457176394031475
0.84 0.0267933301905215
0.88 0.00514434853782276
0.92 0.0114314083517037
0.96 0.0121867495552139
1 0.00483785140095601
1.04 0.00821064057668169
1.08 0.00781546717759778
1.12 0.00459557155439135
1.16 0.00640831039782046
1.2 0.00582519711787819
1.24 0.00439650444081881
1.28 0.00518274621735259
1.32 0.00459136583962116
1.36 0.0042979602832129
1.4 0.0045351031160176
1.44 0.00425099183733273
1.48 0.00395268294769008
1.52 0.00419993396901107
1.56 0.00371439356752287
1.6 0.00361547282302108
1.64 0.00347112930377825
1.68 0.00333205620788255
1.72 0.00313650015528758
1.76 0.0030114611416993
1.8 0.0029187213148796
1.84 0.00271799980181833
1.88 0.0026534749807325
1.92 0.00247737761928337
1.96 0.00229452592506619
2 0.00228465753320237
};
\addlegendentry{data}
\end{axis}

\end{tikzpicture}
		\caption{extension}
		\label{sbfig:two_kernels:observations:extension}
	\end{subfigure}
	\caption{Evolution of the tip displacement comparing the ``true" model and the model calibrated from the noisy data with the noise level $2\%$, using two loading types: (\subref{sbfig:two_kernels:observations:bending}) - bending, 
	(\subref{sbfig:two_kernels:observations:extension}) - extension.
	}
	\label{fig:two_kernels:observations}
\end{figure}
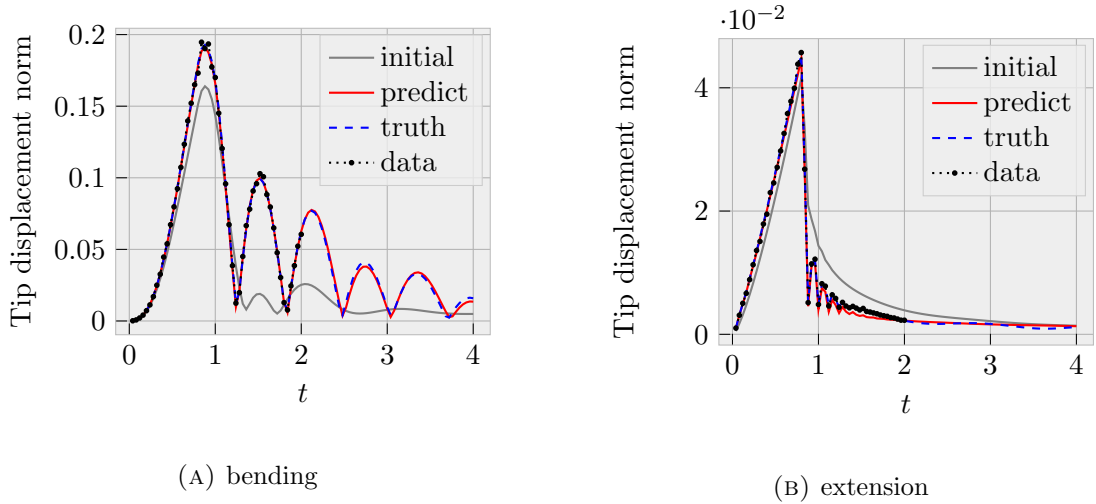

~\\
\newpage
~\\
\begin{figure}[!h]
	\begin{subfigure}{0.45\textwidth}
		\input{plt_two_kernels_compare_kernels_tr.tex}
		\caption{hydrostatic kernel}
	\end{subfigure}
	\hfill
	\begin{subfigure}{0.45\textwidth}
		\input{plt_two_kernels_compare_kernels_dev.tex}
		\caption{deviatoric kernel}
	\end{subfigure}
	\caption{Comparison of the resulting kernels $\kertreps(t)$ (left) and $\kereps(t)$ (right) in the log-scale.}
	\label{fig:two_kernels:kernels}
\end{figure}

\begin{appendices}   
\section{Appendix: Regularity} \label{Appendix1}
\noindent We briefly discuss here the higher regularity of the displacement $\bfu$ in a setting where $\GN=\emptyset$ and $\partial \Omega$ sufficiently smooth. More precisely, we present the formal derivation of a higher-order estimate in terms of $\bfz$ and note that additional bounds on $\bfu$ can be obtained by distinguishing the smooth and singular cases for $\kersig$, as in Propositions~\ref{Pro:state:exis_uni1} and~\ref{Pro:state:exis_uni2}. \\
\indent Abbreviating $\boldsymbol{w}=\bbC \epsi(\bfz)$, $\boldsymbol{v}={\KerTen}*\epsi(\bfu_t)$ and testing 
\begin{equation}
\begin{aligned}
\rho\bfz_{tt}-\div\left[\boldsymbol{w} +\boldsymbol{v}\right] =\bfg
\end{aligned}
\end{equation}
 with $-\div(\boldsymbol{w}_t+\boldsymbol{v}_t)$ after integration over $\Omega\times(0,t)$ yields
\[
\begin{aligned}
&\begin{multlined}[t]\frac12\|\varrho^{1/2}\bbC^{1/2}\epsi(\bfz_t(t))\|_{L^2(\Omega)^d}^2 + \frac12\|\div(\boldsymbol{w}+\boldsymbol{v})(t)\|_{L^2(\Omega)^d}^2\\
+\int_0^t\int_\Omega \left( (\kersig*\epsi(\bfu_{t})_{tt}+ \epsi(\bfu_{tt})\right)
{\KerTen}*\epsi(\bfu_{tt})\dxt \end{multlined}
\\
=&\, -\int_0^t\int_\Omega \bfg_t(s)\div(\boldsymbol{w}+\boldsymbol{v})(s)\dxs + \int_\Omega \bfg(t)\div(w+v)(t)\dx\\
\leq&\,\begin{multlined}[t] 2\|\bfg_t\|_{L^1(0,t;L^2(\Omega)^d)}+\frac18\|\div(\boldsymbol{w}+\boldsymbol{v})\|_{L^\infty(0,t;L^2(\Omega)^d)}^2
+2\|\bfg(t)\|_{L^2(\Omega)^d}^2\\ +\frac18\|\div(\boldsymbol{w}+\boldsymbol{v})(t)\|_{L^2(\Omega)^d}^2,\end{multlined}
\end{aligned}
\]
where we have used $\boldsymbol{w}(t=0)=\boldsymbol{v}(t=0)=0$. Note that the second and fourth term on the right-hand side can be absorbed into the second term on the left-hand side. Due to the assumptions on $\KerTen$, we have
\begin{equation}
\int_0^t\int_\Omega \left( (\kersig*\epsi(\bfu_{t})_{tt}+ \epsi(\bfu_{tt})\right)
{\KerTen}*\epsi(\bfu_{tt})\dxt \geq 0.
\end{equation}
Altogether, we obtain the estimate
\[
\begin{aligned}
\|\epsi(\bfz_t(t))\|_{L^2(\Omega)^d}^2 + \|\div(\boldsymbol{w}+\boldsymbol{v})(t)\|_{L^2(\Omega)^d}^2
\lesssim\, \|\bfg_t\|_{L^1(0,t;L^2(\Omega)^d)}+\|\bfg(t)\|_{L^2(\Omega)^d}^2.
\end{aligned}
\]
a.e.\ in time.
\section{Appendix: Auxiliary inequalities}
Analogously to~\cite[Lemma 1]{Alikhanov:11}, one can prove the following lower bound.
\begin{lemma}\label{lem:Alikhanov1}
Given Hilbert space $X$ and any $w\in H^1(0,T;H)$, $k\in W^{1,1}(0,T)$ with $k\geq0$, $k'\leq0$ a.e.,
\begin{equation}
\int_0^T\langle w(t), k* w_t(t)\rangle\dt \geq 
\frac12(k* \|w\|^2)(T)-\frac12\int_0^Tk(t)\dt\, \|w(0)\|^2\,.
\end{equation}
\end{lemma}
\begin{proof}
	The statement follows from
\[
\begin{aligned}
&\langle w(t),(k* w_t)(t)\rangle - \frac12 (k*\|w\|^2_t)(t)
=\int_0^t \left\langle w(t)- w(s), k(t-s) w_t(s)\right\rangle \ds \\
&=\int_0^t\langle \int_s^t w_t(r)\dr ,k(x-s) w_t(s)\rangle \ds 
=\int_0^t\int_0^r \langle w_t(r), k(t-s) w_t(s)\rangle  \dsr\\
&=\int_0^t\frac{1}{k(t-r)}k(t-r) \langle w'(r), \int_0^rk(t-s) w_t(s)\ds\rangle  \dr\\
&=\frac12\int_0^t\frac{1}{k(t-r)}\frac{d}{dr}\Bigl\|\int_0^rk(t-s) w_t(s)\ds\Bigr\|^2\dr\\
&=-\frac12\int_0^t\frac{k'(x-r)}{k^2(x-r)}\Bigl\|\int_0^rk(t-s) w_t(s)\ds\Bigr\|^2 \dr 
+ \frac12 \left[\frac{1}{k(t-r)}\Bigl\|\int_0^rk(t-s) w'(s)\ds\Bigr\|^2 \right]_0^t\\
&\geq0.
\end{aligned}
\]
\end{proof}
Similarly to~\cite[Lemma 2.3]{Eggermont:1987} (see also~\cite[Theorem 1]{VoegeliNedaiaslSauter:2016}) one obtains the following coercivity estimate.
\begin{lemma}\label{lem:coercivityI}
For some Hilbert space $X$ and any $w\in H^{-\delta/2}(0,T;H)$, $k\in L^1(0,T)$ with $\Re(\mathcal{F}k)(\omega)\geq\gamma(1+\omega^2)^{-\delta/2}$, $\omega\in\mathbb{R}$,
	\begin{equation}\label{eqn:coercivityI}
	\int_0^T \langle w(s),k* w(s)\rangle \ds \geq \gamma \| w \|_{H^{-\delta/2}(0,T;X)}^2, 
	\end{equation}	 
\end{lemma}
\begin{proof}
We first assume $w\in L^2(0,T;X)$ and approximate 
$k*w$  by $k_\epsilon*w$ with $k_\epsilon(t)=k_\epsilon(t) e^{-\epsilon t}$ 
for $\epsilon>0$ and extend $w\in L^2(0,T)$ to all of $\mathbb{R}$ by zero.
Plancherel's Theorem and the Convolution Theorem then yield
\[
\begin{aligned}
\int_0^t  \langle w(s), (k_\epsilon* w)(s)\rangle \ds 
&=\int_\mathbb{R} \langle\mathcal{F}w(\omega), \overline{\mathcal{F}[k_\epsilon*w]}(\omega)\rangle\, \textup{d}\omega\\
&=\int_\mathbb{R} \|\mathcal{F}w(\omega)\|^2 \overline{\mathcal{F}k_\epsilon}
(\omega)\, \textup{d}\omega\,.
\end{aligned}
\]
Here both sides have to be real valued, because obviously the left hand side is, and therefore
\[
\begin{aligned}
\int_0^t  \langle w(s), (k_\epsilon* w)(s)\rangle \ds 
=\int_\mathbb{R} \|\mathcal{F}w(\omega)\|^2 \Re(\mathcal{F}k_\epsilon)(\omega)\, \textup{d}\omega
\end{aligned}
\]
Letting $\epsilon$ tend to zero and recalling that 
$\| w \|_{H^{-\delta/2}(0,x)}^2=\int_\mathbb{R} |\mathcal{F}w(\omega)|^2 (1+\omega^2)^{-\delta/2}\, \textup{d}\omega$, we can conclude the inequality \eqref{eqn:coercivityI} for $w\in L^2(0,T;X)$.
Now the assertion follows for general $w\in H^{-\delta/2}(0,T;X)$ by density.
\end{proof}

\end{appendices} 
\section*{Acknowledgement}
\sloppy The authors UK and BW gratefully acknowledge the support of the Deutsche
Forschungsgemeinschaft (DFG) within the project \mbox{WO~671/11-1}.
The work of BK was supported by the Austrian Science Fund FWF under the grants P30054 and DOC 78. MLR acknowledges the support by the Laura Bassi Postdoctoral Fellowship (Technical University of Munich).

\end{document}